\documentclass[11pt,oneside]{amsart}

\title[Equivalence relations on $\NN$ under computable
reducibility]{The hierarchy of equivalence relations on the natural
  numbers under computable reducibility}

\author{Samuel Coskey}
\address{Samuel Coskey, York University \& The Fields Institute, 222
  College Street, Toronto, ON M5S 2N2, Canada}
\email{scoskey@nylogic.org}
\urladdr{boolesrings.org/scoskey}

\author{Joel David Hamkins}
\address{Joel David Hamkins, Department of Philosophy, New York
  University, 5 Washington Place, New York, NY 10003 \& The Graduate
  Center of the City University of New York, Mathematics Program, 365
  Fifth Avenue, New York, NY 10016, \& College of Staten Island of
  CUNY, 2800 Victory Blvd., Staten Island, NY 10314}
\email{jhamkins@gc.cuny.edu}
\urladdr{jdh.hamkins.org}

\author{Russell Miller}
\address{Russell Miller, The Graduate Center of The City University of
  New York, Mathematics Program, 365 Fifth Avenue, New York, NY 10016
  \& Queens College of CUNY, 65-30 Kissena Blvd., Flushing, NY 11367}
\email{Russell.Miller@qc.cuny.edu}
\urladdr{qcpages.qc.cuny.edu/$\sim$rmiller}

\thanks{The research of the first author is partially supported by the
  Natural Sciences and Engineering Research Council of Canada.  The
  research of the second author has been partially supported by grants
  from the National Science Foundation, the Simons Foundation and the
  CUNY Research Foundation. The research of the third author has been
  partially supported by the Isaac Newton Institute and
  by grants from the National Science Foundation
  and the CUNY Research Foundation.}

\usepackage[marginratio=1:1]{geometry}
\usepackage{mydefs}
\usepackage{mathpazo}
\usepackage{setspace}
\usepackage{tikz}

\renewcommand{\subset}{\subseteq}

\newcommand{\smin}{{\sf min}}
\newcommand{\smax}{{\sf max}}
\newcommand{\sgcd}{{\sf gcd}}
\newcommand{\slcm}{{\sf lcm}}
\newcommand{\smed}{{\sf med}}
\newcommand{\sset}{{\sf set}}
\newcommand{\sbin}{{\sf bin}}
\newcommand{\stern}{{\sf tern}}
\newcommand{\sgraph}{{\sf graph}}
\newcommand{\slo}{{\sf lo}}
\newcommand{\stree}{{\sf tree}}
\newcommand{\sgroup}{{\sf group}}
\newcommand{\spres}{{\sf pres}}
\newcommand{\la}{\langle}
\newcommand{\ra}{\rangle}
\newcommand{\bfd}{\boldsymbol{d}}

\newcommand{\lcm}{\mathop{\mathrm{lcm}}}
\newcommand{\med}{\mathop{\mathrm{med}}}
\newcommand{\cut}{\mathop{\mathrm{cut}}\nolimits}

\newcommand{\ciso}{\simeq}
\newcommand{\ociso}{\mathord{\simeq}}

\begin{document}
\begin{abstract}
  The notion of computable reducibility between equivalence relations
  on the natural numbers provides a natural computable analogue of
  Borel reducibility.  We investigate the computable reducibility
  hierarchy, comparing and contrasting it with the Borel reducibility
  hierarchy from descriptive set theory.  Meanwhile, the notion of
  computable reducibility appears well suited for an analysis of
  equivalence relations on the c.e.\ sets, and more specifically, on
  various classes of c.e.\ structures.  This is a rich context with
  many natural examples, such as the isomorphism relation on c.e.\
  graphs or on computably presented groups.  Here, our exposition
  extends earlier work in the literature concerning the classification
  of computable structures.  An abundance of open questions remains.
\end{abstract}
\maketitle

\section{Introduction}

In this paper we aim to study the complexity of equivalence relations
on the natural numbers and their hierarchy, under the relation of
computable reducibility. Using a blend of methods from computability
theory and descriptive set theory, we thus carry out a computable
analogue of the theory of Borel equivalence relations. The resulting
theory appears well suited for undertaking an analysis of the
complexity of isomorphism relations and other equivalence relations on
the class of computably enumerable (c.e.) structures, such as the
isomorphism relation on c.e.~graphs or on computably presented groups
and the orbit equivalence relations arising from computable group
actions. The c.e.~analogues of many of the relations playing important
roles in the Borel theory, such as equality, $E_0$ and others, play
similar roles in the computable theory here, and in addition, the
naturally arising equivalence relations from computability theory on
c.e.~sets, such as Turing equivalence, 1-equivalence and
m-equivalence, also fit into the hierarchy. Although in broad strokes
the resulting theory exhibits many of the important and attractive
features of the Borel theory, there are notable differences, such as
the existence of a complex hierarchy of relations strictly below the
equality relation on c.e.~sets, as well as relations incomparable to
equality, and we shall remark on these differences when they arise.

In the subject known as Borel equivalence relations, one studies
arbitrary equivalence relations on standard Borel spaces with respect
to \emph{Borel reducibility}, a notion introduced in \cite{friedman}.
Here, if $E,F$ are equivalence relations on $X,Y$, then we say that
$E$ is Borel reducible to $F$, written $E\leq_BF$, if there is a Borel
function $f\from X\into Y$ such that
\[x\mathrel{E}x'\iff f(x)\mathrel{F}f(x')\;.
\]
This notion is particularly meaningful when $E$ or $F$ represent a
naturally arising classification problem in mathematics, for it allows
us to compare the difficulty of such classification problems in a
precise and robust manner. For instance, the isomorphism relations on
the spaces of countable groups, graphs, fields, and linear orders, the
isometry relation on separable Banach spaces, and the conjugacy
relation on measure-preserving transformations can all be compared
with respect to Borel reducibility.  In these cases, we interpret
$E\leq_BF$ as saying that the classification problem for elements of
$X$ up to $E$ is \emph{no harder than} the classification problem for
elements of $Y$ up to $F$.  Indeed, the reduction function $f$ yields
a classification of the elements of $X$ up to $E$ using invariants
from $Y/F$.  For a fantastic introduction to the subject and its
motivations we recommend the text \cite{gao}, and we will often cite
it for background material.

Motivated by the desire to imitate this field of research in the
computable setting, we work with the following notion of
reducibility for equivalence relations on the natural numbers.

\begin{defn}
  \label{def:main}
  Let $E,F$ be equivalence relations on $\NN$.  We say that $E$ is
  \emph{computably reducible} (or just \emph{reducible}) to $F$,
  written $E\leq F$, if there exists a computable function
  $f\from\NN\into\NN$ such that
  \[n\mathrel{E}n'\iff f(n)\mathrel{F}f(n')
  \]
  In other words, the equivalence classes $\NN/F$ form a set of
  effectively computable invariants for the classification problem up
  to $E$.
\end{defn}

As evidence that this is a very natural definition, we remark that a
substantial number of other authors have arrived at this notion from
other directions.  (Our own investigation of computable reducibility
began independently from this prior work as well.)  Without intending
to give a completely history, we now briefly discuss this notion's
travels through the literature.  It seems that computable reducibility
first appeared, at least in the English language, in \cite{sorbi},
where it is defined only for c.e.\ equivalence relations.  The authors
of that paper note that the notion was in use prior to their work, for
instance in \cite{ershov}.  The theory of c.e.\ equivalence relations
with respect to computable reducibility was later expanded in
\cite{gerdes} and then \cite{sorbi:ceers}.  In an independent line of
study, Knight and coauthors considered a number of
computability-theoretic notions of reducibility between various
natural classes of structures (see the series of papers
\cite{knight:comparison,knight:classification,knight:embeddings}).
The notion of computable reducibility again arose in
\cite{fokina:computable} and \cite{knight:iso}, where Fokina, Friedman
and others used it to compare classes of computable structures.  In
\cite{fokina:sigma}, the authors used the notion to compare classes of
hyperfinite structures, as well as arbitrary equivalence relations.

Meanwhile, the number of names given to the reducibility of
Definition~\ref{def:main} is almost as large as the number of papers
about it, and our name continues this tradition.  The original name of
$m$-reducibility, given by Bernardi and Sorbi and kept by Gao and
Gerdes, reflects the analogy to $m$-reductions in computability theory
(see e.g. \cite{soare}), but its use can be confusing, since one
equivalence relation can be $m$-reducible to another in the sense of
\cite{sorbi} yet not in the sense of \cite{soare}, or vice versa.
Knight and her coauthors actually used two distinct names for the same
concept, since they viewed these as relations on classes of
structures, not as reducibilities on equivalence relations.  Finally,
the notion is called FF-reducibility in
\cite{fokina:computable,knight:iso,fokina:sigma}.  We believe that
\emph{computable reducibility} serves our purposes best, since we are
motivated by the analogy to Borel reducibility on equivalence
relations.  Borel reductions from $E$ to $F$ are Borel maps $f$
satisfying $xEy\iff f(x)Ff(y)$, and our reductions are
Turing-computable maps $f$ with the same property.  (We would gladly
have called it ``Turing reducibility,'' were that name not already in
use.)  Likewise, it would be natural to study $\bfd$-computable
reducibility, using $\bfd$-computable functions $f$, for arbitrary
Turing degrees $\bfd$, or to study $\mathcal{C}$-reducibility for
other classes $\mathcal{C}$ of functions.

In this paper, we shall apply Definition~\ref{def:main} to the study
of several varieties of equivalence relations.  The simplest of these
are the partitions of the natural numbers which are of a
number-theoretic or combinatorial nature.  Here, the notion of
computable reducibility can be seen as a degree-theoretic structure on
the equivalence relations which properly generalizes the classical
Turing reducibilities.  The connection with degree theory is more than
just an analogy; for instance, we will observe in
Proposition~\ref{prop:manyone} that the $1$-reducibility ordering on
the c.e.\ $1$-degrees embeds into the computable reducibility ordering
on the equivalence relations with two classes.

Another connection with degree theory is the role of the arithmetical
hierarchy.  Here our theory departs from the classical Borel theory in
a significant way.  For instance, in the Borel theory almost all
interesting equivalence relations are either Borel or $\Sigma^1_1$. As
a consequence, equivalence relations cannot typically be distinguished
up to Borel bireducibility just on the basis of their position in the
projective hierarchy.  In this paper, we shall consider equivalence
relations on various levels of the arithmetic hierarchy, and this will
give us a convenient and powerful tool for establishing
nonreducibility.

A second, more substantial variety of equivalence relations, to which
we shall devote most of our attention, are those arising from
relations on the collection of computably enumerable (c.e.)\ subsets
of $\NN$. The goal is to study the hierarchy of equivalence relations
arising naturally in this realm---including isomorphism relations on
natural classes of c.e.~structures, such as groups, graphs and
rings---in a manner analogous to the Borel theory. Since the c.e.~sets
and structures have a canonical enumeration $\{W_e\}_{e\in\NN}$ from
computability theory, every equivalence relation on the c.e.~sets
arises from a corresponding equivalence relation on the indices for
those sets, in effect using the program index $e$ to stand in for the
set $W_e$ that it enumerates. Specifically, when $E$ and $F$ are
equivalence relations defined on the c.e.\ sets, we can say that
$E\leq F$ if and only if there is a computable function $f$ such that
for all indices $e,e'$,
\[W_e\mathrel{E}W_{e'}\iff W_{f(e)}\mathrel{F}W_{f(e')}\;.
\]
Thus, $E\leq F$ in this sense if and only if $E^{ce}\leq F^{ce}$ in
the sense of Definition~\ref{def:main}, where $E^{ce}$ denotes the
relation on $\NN$ defined by $e\mathrel{E}^{ce}e'\iff
W_e\mathrel{E}W_{e'}$.

In this context, computable reducibility is more closely analogous
with the Borel theory.  For instance, many of the classically studied
equivalence relations can be fruitfully restricted to just the c.e.\
sets. Moreover, many of the Borel reductions between these relations
turn out to be computable in our sense, yielding a familiar hierarchy
of equivalence relations on c.e.\ sets.  As we shall see, there is
even an analogue of the very important class of countable Borel
equivalence relations.

The last type of equivalence relations that we shall consider are
isomorphism relations, that is, equivalence relations which arise from
classification problems.  One might expect that we would be interested
in the isomorphism relation on finite structures, since these are
coded by natural numbers.  However, it is computable whether two
finite structures are isomorphic, and we shall see that computable
equivalence relations are essentially trivial according to our
reducibility notion.  Instead, we shall consider isomorphism of
\emph{c.e.\ structures} by again using the indices as stand-ins for
the structures they code.  In this final context, our efforts either
extend or stand in close analogy with the results in recent literature
concerning effective notions of reducibility.

In addition to those already mentioned, many other important analogues
of Borel reducibility theory appear in the literature.  For instance
the effective version of Borel reducibility, that is, hyperarithmetic
reducibility between equivalence relations on the real or natural
numbers, is studied in \cite{asger:effective}.  The even weaker notion
of PTIME reducibility is considered in \cite{buss}.  Finally, in
\cite{ittm} the authors consider a computability-theoretic
strengthening of Borel reducibility, namely the reductions which are
computable by an infinite time Turing machine.

This paper is organized as follows.  In the next section we consider
relations of the first variety, that is, relations on the natural
numbers taken at face value.  Here, we show that computable
reducibility of equivalence relations in some sense generalizes both
many-one and one-one reducibility of c.e.\ sets.  In the third and
fourth sections we begin to export some of the Borel equivalence
relation theory to the theory of equivalence relations on c.e.\ sets,
for instance observing that many of the classical Borel reductions
hold in our context as well.  On the other hand, we show that some
unexpected phenomena occur, such as the existence of a large hierarchy
of relations on c.e.\ sets which lie properly below the equality
relation. In the fifth section we define and discuss c.e.\ analogues
of the countable Borel equivalence relations and orbit equivalence
relations.  In the sixth section we introduce the theory of
isomorphism and computable isomorphism relations on classes of c.e.\
structures. Finally, in the last section we compare equivalence
relations arising from computability theory itself, such as the Turing
degree relation on c.e.\ sets.

\section{Combinatorial relations on $\NN$}

In this section, we present a series of elementary results concerning
relations on $\NN$ of very low complexity.  It should be noted that
most of the statements here can be found in the aforementioned
literature (see for instance \cite{sorbi},\cite{gerdes}, and
\cite{fokina:sigma}).

We begin at what is for us the very lowest level, where some easy
general observations lead quickly to a complete classification of the
computable equivalence relations up to bireducibility.  The fact that
the computable relations are essentially trivial in the hierarchy of
computable reducibility stands in contrast to the Borel theory, of
course, where the Borel equivalence relations have a wildly rich
structure under Borel reducibility.  Meanwhile, at a level just above
the computable relations, we show that the hierarchy immediately
exhibits enormous complexity in the context of c.e.~equivalence
relations.  Here, we mention just a few results; a much more detailed
exposition of the relations at the c.e.\ level can be found in
\cite{gerdes}.  In later sections, we shall treat many natural
equivalence relations arising from much higher realms of the
arithmetic and even the descriptive set-theoretic hierarchies.

\begin{defn}\
  \begin{itemize}
  \item For each $n$ let $=_n$ be the equality relation on
    $0,\ldots,n-1$.  In order to make the relation defined on all of
    $\NN$, we throw the remaining numbers $n,n+1,\ldots$ into the
    equivalence class of $n-1$.
  \item Let $=_\NN$ denote the equality relation on $\NN$.
  \end{itemize}
\end{defn}

Thus $=_n$ is a canonically defined computable equivalence relation on
$\NN$ with exactly $n$ equivalence classes.  Similarly, $\NN$, is a
canonical computable equivalence relation with infinitely many
classes.

\begin{prop}\
  \label{prop:min}
  \begin{itemize}
  \item If $E$ is any equivalence relation with at least $n$ classes,
    then $=_n$ is reducible to $E$.
  \item If $E$ is $\Pi^0_1$ and $E$ has infinitely many equivalence
    classes, then $=_\NN$ is computably reducible to $E$.
  \end{itemize}
\end{prop}

\begin{proof}
  Let $i_0,\ldots,i_{n-1}$ be a system of pairwise $E$-inequivalent
  natural numbers.  Then $=_n$ is easily seen to be reducible to $E$
  by the map $f(k)=i_k$ for $k< n$, and $f(k)=i_{n-1}$ for $k\geq n$.

  Now suppose that $E$ is $\Pi^0_1$ and that $E$ has infinitely many
  classes.  Then $=_\NN$ is reducible to $E$ by the map which, on
  input $n$, begins enumerating $E$-incomparable elements.  When a
  system of $n$ pairwise $E$-incomparable elements is found, we map
  $n$ to the largest one.
\end{proof}

On the other hand, we shall see later on that there exist c.e.\
equivalence relations which are computably incomparable with $=_\NN$.

\goodbreak
\begin{prop}\label{prop:comp}\
  \begin{itemize}
  \item If $E$ is a $\Sigma^0_1$ or $\Pi^0_1$ equivalence relation
    with finitely many equivalence classes, then $E$ is computable.
  \item If $E$ is computable and has exactly $n$ classes, then $E$ is
    computably reducible to $=_n$.
  \item If $E$ is computable, then $E$ is computably reducible to
    $=_\NN$.
  \end{itemize}
\end{prop}

\begin{proof}
  Begin by letting $i_1,\ldots,i_n$ be a maximal system of pairwise
  $E$-inequivalent natural numbers.  If $E$ is $\Sigma^0_1$, then
  given natural numbers $a,b$ we can decide whether $a\mathrel{E}b$ as
  follows.  Begin enumerating $E$-equivalent pairs until it is
  discovered that $a\mathrel{E}i_{j_1}$ and $b\mathrel{E}i_{j_2}$.
  Then $a\mathrel{E}b$ if and only if $j_1=j_2$.

  On the other hand, if $E$ is $\Pi^0_1$ then we can enumerate
  $E$-inequivalent pairs until we find that $a\not\mathrel{E}i_j$ for
  all $j$ other than some $j_1$, and $b\not\mathrel{E}i_j$ for all $j$
  other than some $j_2$.  Then again, $a\mathrel{E}b$ if and only if
  $j_1=j_2$.

  Finally, if $E$ is computable then given $a$, a program can order
  the equivalence classes which occur below $a$ by their least
  elements, and map $a\mapsto j$ if $a$ is in the $j\th$ class.  Note
  that in the case that $E$ has just $n$ classes, then this is also a
  reduction to $=_n$.
\end{proof}

We thus obtain a complete classification of the computable equivalence
relations by the number of equivalence classes.  This situation is
identical to the classification of Borel equivalence relations with
just countably many classes up to Borel reducibility.

\begin{cor}
  An equivalence relation $E$ is computable if and only if it is
  computably bireducible with one of $=_n$ or $=_\NN$.
\end{cor}

The above results show that if we want to find distinct equivalence
relations with some finite number of classes $n$, then we must look at
least to the level $\Delta^0_2$.  In fact, we need look no higher,
since for instance any $\Delta^0_2$ non-computable equivalence
relation with exactly $n$ classes is not computably reducible to
$=_n$.  We now observe that there is significant complexity even among
the $\Delta^0_2$ equivalence relations with just two classes.

\begin{defn}
  \label{defn:EAC}
  For $A\subset\NN$, write $A^c$ for $\NN\smallsetminus A$.  Then
  $E_{A,A^c}$ denotes the equivalence relation defined by
  $n\mathrel{E}_{A,A^c}n'$ if and only if both $n,n'\in A$ or neither
  $n,n'\in A$.  That is, $E_{A,A^c}$ has two classes: $A$ and $A^c$.
\end{defn}

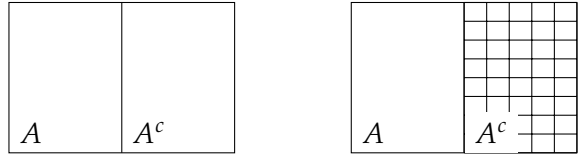
\begin{figure}
\begin{tikzpicture}
\draw (0,0) rectangle (3,2);
\draw (1.5,0) -- (1.5,2);
\node[anchor=south west] at (0,0) {$A$};
\node[anchor=south west] at (1.5,0) {$A^c$};
\end{tikzpicture}\qquad\qquad
\begin{tikzpicture}
\draw (0,0) rectangle (3,2);
\draw (1.5,0) -- (1.5,2);
\draw[xstep=.3,ystep=.25] (1.5,0) grid (3,2);
\node[anchor=south west] at (0,0) {$A$};
\node[anchor=south west,fill=white] at (1.5,0) {$A^c$};
\end{tikzpicture}
\caption{Left: The equivalence relations described in
  Definitions~\ref{defn:EAC} (left) and \ref{defn:EA} (right).}
\end{figure}

It is clear that if $A$ is c.e., then the relation $E_{A,A^c}$ is
$\Delta^0_2$; indeed, its complement is a difference of c.e.\ sets.
The result below follows directly from the definition of many-one
reducibility (see \cite[Definition~I.4.7]{soare}).

\begin{prop}
  \label{prop:manyone}
  $E_{A,A^c}$ is reducible to $E_{B,B^c}$ if and only if $A$ is many-one
  reducible to either $B$ or $B^c$.\qed
\end{prop}

Thus we obtain a copy of the partial ordering of many-one degrees,
modulo the relation which identifies $A$ and $A^c$.

We now turn to a discussion of the c.e.\ equivalence relations.  We
show in particular that the computable reducibility hierarchy has an
interesting and rich structure even at this low complexity level.  The
next result reflects the idea, hinted at in the introduction, that
computable reducibility of equivalence relations in some sense
generalizes the notion of one-one Turing reducibility of sets.

\begin{defn}
  \label{defn:EA}
  For any $A\subset\NN$, let $E_A$ denote the equivalence relation
  defined by $n\mathrel{E}_An'$ if and only if $n,n'\in A$ or $n=n'$.
  That is, $A$ is one equivalence class, and each remaining point is
  in an equivalence class by itself.
\end{defn}

Observe that if $A$ is c.e., then $E_A$ is c.e.

\begin{prop}
  Suppose that $A,B\subset\NN$ are c.e.\ and non-computable.  Then
  $E_A$ is reducible to $E_B$ if and only if $A$ is 1-reducible to
  $B$.
\end{prop}

\begin{proof}
  If $f\from\NN\into\NN$ is a 1-reduction from $A$ to $B$, then it is
  easy to see that $f$ is a also a computable reduction from $E_A$ to
  $E_{B}$.  On the other hand, suppose that $f$ is a reduction from
  $E_A$ to $E_{B}$.  Then $f(A)$ is either contained in $B$, or else
  it is a singleton.  If $f(A)$ were a singleton, say $\{n\}$, then
  $A=f^{-1}(\{n\})$ would be computable, contradicting our hypothesis.
  Hence $f(A)\subset B$ and $f(\NN\smallsetminus
  A)\subset\NN\smallsetminus B$, and so $f$ is a many-one reduction
  from $A$ to $B$.  Moreover, $f$ is already injective on
  $\NN\setminus A$.

  Now, we will adjust $f$ on $A$ to obtain a $1$-reduction $g$ from
  $A$ to $B$ as follows.  Let $g(0)=f(0)$, and inductively let
  $g(n+1)=f(n+1)$ so long as $f(n+1)$ is distinct from
  $g(0),\ldots,g(n)$.  On the other hand, if $f(n+1)=g(k)$ for some
  $k\in0,\ldots,n$, then we must have that $f(n+1)\in B$.  Hence we
  simply enumerate $B$ in search of a new element $a\in B$ which is
  distinct from $g(0),\ldots,g(n)$, and let $f(n+1)=a$.  Then $g$ is
  as desired.
\end{proof}

Thus the partial ordering of c.e.\ equivalence relations is at least
as complicated as the $1$-reducibility ordering on the c.e.\
$1$-degrees.  This last ordering is known to be quite
complex. Moreover, we shall show later on that it contains a copy of
the c.e.\ sets with the partial ordering of containment.

\begin{cor}
  There exist c.e.\ relations which are incomparable with $=_\NN$.
\end{cor}

\begin{proof}
  Consider the relation $E_A$ where $A$ is a simple set (see
  \cite[Theorem~V.1.3]{soare}).  That is, $A$ is a c.e.\ co-infinite
  set whose complement contains no infinite c.e.\ sets.  Since $A$ is
  not computable, neither is $E_A$, and it follows that $E_A$ is not
  reducible to $=_\NN$.  On the other hand, if $f$ is a computable
  reduction from $=_\NN$ to $E_A$ then there exists $n\in\NN$ such
  that $f(\NN\smallsetminus\{n\})\subset A^c$.  But this is a
  contradiction, since $f(\NN\smallsetminus\{n\})$ is c.e.\ and
  infinite.
\end{proof}

We next show that the class of c.e.\ equivalence relations, for all of
its complexity, still admits a universal element.  Again, this
reflects the situation for the c.e.\ degrees.

\begin{prop}
  There exists a c.e.\ relation $U_{ce}$ which is universal in the
  sense that every c.e.\ equivalence relation is reducible to
  $U_{ce}$.  Moreover, for any real parameter $z$, there exists an
  equivalence relation which is universal for all equivalence
  relations which are c.e.\ in $z$.
\end{prop}

\begin{proof}
  For any program $e$, let $E_e$ be the equivalence relation obtained
  by taking the transitive closure of whatever relation is enumerated
  by $e$, together with the diagonal for reflexivity.  That is, we
  interpret the set $W_e$ as pairs, and we set $E_e$ to be the
  smallest equivalence relation containing these pairs. Thus, $E_e$ is
  the $e\th$ c.e.\ equivalence relation, and every c.e.\ equivalence
  relation arises this way.

  Now define the universal relation by $(e,a)\mathrel{U}_{ce}(e',a')$
  if and only if $e=e'$ and $a\mathrel{E_e}a'$.  Thus, we have divided
  $\NN$ into slices, and put $E_e$ on the $e\th$ slice. This relation
  is c.e., since we may computably enumerate approximations to
  $U_{ce}$ by running all programs and taking the transitive closure
  of what has been produced so far.  It is universal for c.e.\
  relations since $E_e$ reduces to $U_{ce}$ by mapping
  $a\mapsto(e,a)$.

  This construction generalizes to oracles as follows. If $z$ is any
  oracle, we have the notion $E_e^z$ and we can define the relation
  $U_{ce}^z$ defined by performing the above construction relative to
  $z$.  The relation $U_{ce}^z$ is $z$-c.e., and every $E_e^z$
  computably reduces to $U_{ce}^z$ (without need for $z$) by the map
  $a\mapsto(e,a)$, as before.
\end{proof}

Many of the results presented in this section so far are summarized in
Figure~\ref{fig:combinatorial}.

\begin{figure}[ht]
\begin{tikzpicture}
  \draw[dotted] (0,2.5) circle (2.5);
  \draw[dotted] (0,1.5) circle (1.5);
  \node[fill=white] at (2.3,1.5) {\small\emph{c.e.}};
  \node[fill=white] at (1.3,1) {\small\emph{comp}};
  \node[fill=white] at (0,0) (1) {$=_1$};
  \node at (0,1) (2) {$=_2$} edge (1);
  \node[fill=white] at (0,3) (N) {$=_\NN$}
    edge node[fill=white,sloped] (dots) {$\cdots$} (2);
  \node[fill=white] at (0,5) (uce) {$U_{ce}$} edge (N);
  \node at (0,7) (eq) {$=^{ce}$} edge (uce);
  \node at (1.5,3.5) (ea) {$E_A$} edge (uce) edge (1) edge (2) edge (dots);
  \node at (-3,4) {$E_{A,A^c}$} edge (eq) edge (2);
\end{tikzpicture}
\caption{Diagram of reducibility among equivalence relations from
  Section~2.  The relation $=^{ce}$ is equality of
  c.e.\ sets as a relation on indices, namely, $e=^{ce}f\iff W_e=W_f$, and
  it will be a major focus of section~3.\label{fig:combinatorial}}
\end{figure}
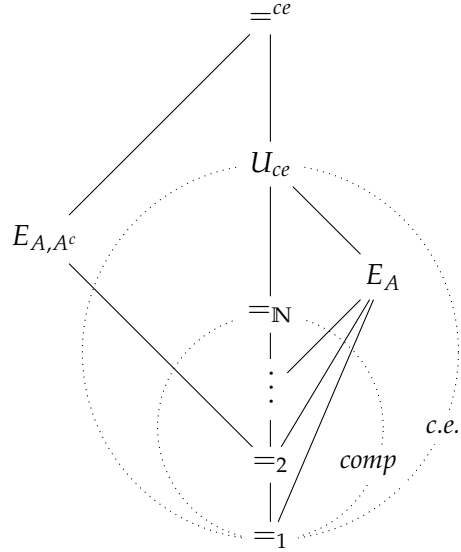

We close this section with a result that is perhaps just a curiosity,
but will motivate our discussion of orbit equivalence relations in
later sections.  In what follows, a group $\Gamma$ is said to be
computable if its domain is $\NN$ and its multiplication function is a
computable set of triples.  If $\Gamma$ is such a group, then a
computable action of $\Gamma$ on $\NN$ is just a computable function
$\Gamma\times\NN\into\NN$ satisfying the usual group action
laws. Finally, if $\Gamma$ acts computably on $\NN$, then the
resulting orbit equivalence relation is defined by
$x\mathrel{E_\Gamma}x'$ if and only if there exists $\gamma\in\Gamma$
such that $x'=\gamma x$.

\begin{thm}
  The c.e.\ equivalence relations are precisely the orbit equivalence
  relations induced by computable actions of computable groups.
\end{thm}

\begin{proof}
  Clearly if $E$ is the orbit equivalence relation induced by a
  computable action, then it is c.e.~ Conversely, suppose that $E$ is
  a c.e.\ relation, enumerated by a program $e$.  Let $\Gamma$ be a
  fixed computable copy of the free group $F_\omega$ with generators
  $x_1,x_2,\ldots$.  If $e$ enumerates a pair $(n,n')$ into $E$ at
  stage $s$, then we let $i$ be a code for the triple $(s,n,n')$ and
  let $x_i$ act by swapping $n$ and $n'$ and leaving the other generators
  fixed.  Clearly the orbit relation arising from this action is
  precisely $E$.  Moreover the action is computable, since each
  generator codes a bound which tells how long to run $e$ to find out
  which elements it swaps.
\end{proof}

\section{Equivalence relations on c.e.\ sets}

We now begin our study of the second variety of equivalence relations,
namely, those defined on the c.e.\ real numbers.  In this section, we
will study a series of key equivalence relations from Borel relation
theory.  As was discussed in the introduction, the c.e.\ sets will be
represented by their indices.

\begin{defn}
  For any any equivalence relation denoted $E$ on the collection of
  c.e.\ subsets of $\NN$, we shall denote by $E^{ce}$, adding the
  superscript ``$ce$,'' the corresponding equivalence relation on
  $\NN$ defined on the indices by $e\mathrel{E}^{ce}e'$ if and only if
  $W_e\mathrel{E}W_{e'}$.
\end{defn}

One of the most important equivalence relation on c.e.\ sets is, of
course, the equality relation.

\begin{prop}
  \label{prop:eqce}
  $=^{ce}$ is a $\Pi^0_2$-complete set of pairs.
\end{prop}

\begin{proof}
  To see that $=^{ce}$ is in $\Pi^0_2$, note that $W_e=W_{e'}$ if and
  only if for all numbers $n$ enumerated into $W_e$, there is a stage
  by which $n$ is also enumerated into the other $W_{e'}$ (and vice
  versa).  To show that it is $\Pi^0_2$ complete, we use the fact that
  the set $\mathrm{TOT}$ of programs which halt on all inputs is
  $\Pi^0_2$ complete (see \cite[Theorem~IV.3.2]{soare}).  The set
  $\mathrm{TOT}$ is $m$-reducible to the $=^{ce}$-equivalence class of
  $W_{e_0}=\NN$ by the following function: $f(e)$ is the program which
  outputs $n$ just in case $e$ halts on input $n$.
\end{proof}

\begin{prop}
  If $E$ is a c.e.\ equivalence relation, then $E$ lies properly below
  $=^{ce}$.
\end{prop}

\begin{proof}
  Let $E$ be an arbitrary c.e.\ relation.  Since each equivalence
  class $[n]_E$ of $E$ is c.e., we can define a reduction from $E$ to
  $=^{ce}$ by mapping each $n$ to a program $f(n)$ which enumerates
  $[n]_E$.  On the other hand, there can't be a reduction from
  $=^{ce}$ to $E$ since $=^{ce}$ is $\Pi^0_2$ complete and $E$ is c.e.
\end{proof}

A second equivalence relation which plays an essential role in the
Borel theory is the \emph{almost equality} relation on $\mathcal
P(\NN)$.  Let $E_0$ denote this relation, that is, $A\mathrel{E}_0B$
if and only if the symmetric difference $A\mathrel{\triangle}B$ is
finite.  Then $E_0^{ce}$ is the almost equality relation on the
(indices for) c.e.\ sets.  We have the following analogue with the
Borel theory.

\begin{thm}
  \label{thm:e0}
  $=^{ce}$ lies strictly below $E_0^{ce}$.
\end{thm}

\begin{proof}
  It is not difficult to build a reduction from $=^{ce}$ to
  $E_0^{ce}$.  For instance, given a program $e$, we can define the
  program $f(e)$ as follows.  Whenever $n$ is enumerated into $W_e$,
  the program $f(e)$ enumerates codes for the pairs
  $(n,0),(n,1),(n,2),\ldots$ (or more accurately, it arranges to
  periodically add more and more of these).  Then clearly we have that
  $W_e$ and $W_{e'}$ differ if and only if $W_{f(e)}$ and $W_{f(e')}$
  differ infinitely often.

  On the other hand, there cannot be a computable reduction from
  $E_0^{ce}$ to $=^{ce}$, since $E_0^{ce}$ is $\Sigma^0_3$ complete
  while $=^{ce}$ is just $\Pi^0_2$.  To see that $E_0$ is $\Sigma^0_3$
  complete, note that by \cite[Corollary~IV.3.5]{soare}, even its
  equivalence class $\mathrm{COF}=\set{e\mid W_e\textrm{ is
      cofinite}}$ is $\Sigma^0_3$ complete.
\end{proof}

It is natural to wonder to what extent the arithmetic equivalence
relations on c.e.\ sets mirror the structure of the Borel equivalence
relations.  For instance, it is a fundamental result of Silver (see
\cite[Theorem~5.3.5]{gao}) that the equality relation $=$ is minimum
among all Borel equivalence relations with uncountably many classes.
Thus, it is natural to ask whether an analogue of Silver's theorem
holds, that is, whether $=^{ce}$ is minimum among some large class of
relations on the c.e.\ sets with infinitely many classes.  In the next
section we shall show that this fails even for relations of very low
complexity.  However, we do not address any other analogues of the
classical dichotomy theorems.

\begin{question}
  Are there any relations lying properly between $=^{ce}$ and
  $E_0^{ce}$?  More generally, is there a form of the Glimm-Effros
  dichotomy in this context?  In other words, is there a large
  collection of equivalence relations $E$ such that if $=^{ce}$ lies
  strictly below $E$, then $E_0^{ce}$ lies below $E$?
\end{question}

We next consider some of the combinatorial equivalence relations that
play key roles in the Borel theory.

\begin{defn}\
  \label{def:ei}
  \begin{itemize}
  \item Let $E_1$ be the equivalence relation on $\mathcal P(\NN)^\NN$
    defined by $(A_n)\mathrel{E}_1(B_n)$ if and only if for almost all
    $n$, $A_n=B_n$.
  \item Let $E_2$ be the equivalence relation defined by
    $A\mathrel{E_2}B$ if and only if $\sum_{n\in A\triangle
      B}1/n<\infty$.
  \item Let $E_3$ be the equivalence relation on $\mathcal P(\NN)^\NN$
    defined by $(A_n)\mathrel{E}_3(B_n)$ if and only if for all $n$,
    $A_n\mathrel{E_0}B_n$.
  \item Let $E_\sset$ be the equivalence relation on $\mathcal
    P(\NN)^\NN$ defined by $(A_n)\mathrel{E_\sset}(B_n)$ if and only if\\
    $\set{A_n\mid n\in\NN}=\set{B_n\mid n\in\NN}$.
  \item Let $Z_0$ denote the \emph{density} equivalence relation
    defined by $A\mathrel{Z}_0B$ if and only if\\
    $\lim\abs{(A\mathrel{\triangle}B)\cap n}/n=0$.
  \end{itemize}
\end{defn}

As usual, we are really interested in the corresponding ``superscript
c.e.''\ relations on the indices.  That is, we define
$e\mathrel{E}_1^{ce}e'$ if and only if $W_e$ and $W_{e'}$, thought of
as subsets of the lattice $\NN\times\NN$, are identical in almost
every column of this lattice. The relations $E_3^{ce}$ and
$E_\sset^{ce}$ are defined analogously, using c.e.\ subsets of
$\NN\times\NN$ to represent c.e.\ sequences of c.e.\ sets.

\begin{prop}
\label{prop:Borel}
  The Borel reductions between the equivalence relations given in
  Definition~\ref{def:ei} hold also in the case of computable
  reducibility.  Specifically, $E_0^{ce}$ is reducible to $E_1^{ce}$,
  $E_2^{ce}$, and $E_3^{ce}$, and $E_3$ is reducible to $E_\sset$ and
  $Z_0$.
\end{prop}

\begin{proof}[Sketch of proof]
  The reductions are just the same as the classical ones from the
  Borel theory.  We will quickly give the reductions so that the
  reader may verify they are computable in our sense.  To begin,
  notice that $E_0$ is reducible to $E_3$ by the map: $A\mapsto
  \NN\times A$, where each column of the image of $A$ looks exactly
  like $A$ itself.  To reduce $E_0$ to $E_1$, map $A$ to $\set{\langle
    x,y\rangle\mid y\geq x~\&~y\in A}$, so that the column $x$ equals
  $A-\{ 0,\ldots,x-1\}$.  For $E_0\leq E_2$, we let $I_n$ be any fixed
  c.e.\ partition of $\NN$ into sets such that $\sum_{i\in
    I_n}1/i\geq1$.  Then it is easy to see that $E_0\leq E_2$ via the
  map $A\mapsto\bigcup_{n\in A}I_n$.

  To show that $E_3\leq Z_0$, we first observe that $E_0$ reduces to
  $Z_0$ via the map $f(A)=\bigcup_{n\in A}[2^n,2^{n+1})$.  (Indeed, if
  $A\mathrel{E_0}B$ then $f(A)\mathrel{\triangle} f(B)$ is finite and
  hence has density zero.  Conversely, if $A\mathrel{E_0}B$ is
  infinite, then $(f(A)\mathrel{\triangle}f(B))\cap N$ will be
  $\geq1/2$ infinitely often, and hence $f(A)\mathrel{\triangle}f(B)$
  does not have density zero.)  Now, fix a partition $I_n$ of $\NN$
  into sets such that the density of $I_n$ is exactly $1/2^n$, and let
  $\pi_n$ denote the unique isomorphism $\NN\iso I_n$.  Then we can
  reduce $E_3$ to $Z_0$ using the map $(A_n)\mapsto\bigcup
  \pi_n(f(A_n))$.

  Finally, to reduce $E_3$ to $E_\sset$, suppose we are given a
  sequence $(A_n)$.  For each column $A_n$ and each $s\in2^{<\NN}$, we
  will place the set $1^n\concat0\concat s\concat(A\smallsetminus
  \abs{s})$ as a column of $f((A_n))$.  It is not difficult to check
  that this mapping is as desired.
\end{proof}

\begin{figure}[ht]
\begin{tikzpicture}
  \node at (0,0) (eq) {$=$};
  \node at (0,1) (e0) {$E_0$} edge (eq);
  \node at (-1,2) (e1) {$E_1$} edge (e0);
  \node at (0,2) (e2) {$E_2$} edge (e0);
  \node at (1,2) (e3) {$E_3$} edge (e0);
  \node at (.5,3) (eset) {$E_{set}$} edge (e3);
  \node at (1.5,3) (z0) {$Z_0$} edge (e3);
\end{tikzpicture}
\qquad\qquad
\begin{tikzpicture}
  \node at (0,0) (eq) {$=^{ce}$};
  \node at (0,1) (e0) {$E_1^{ce}\sim E_0^{ce}$} edge (eq);
  \node at (-1,2) (e2) {$E_2^{ce}$} edge (e0);
  \node at (1,2) (e3) {$E_3^{ce}$} edge (e0);
  \node at (.5,3) (eset) {$E_{set}^{ce}$} edge (e3);
  \node at (1.5,3) (z0) {$Z_0^{ce}$} edge (e3);
\end{tikzpicture}
\caption{Left: The Borel reducibility relations between classical
  combinatorial equivalence relations.  This diagram is complete for
  Borel reducibility.  Right: Computable reducibility relations
  between c.e.\ versions of these relations.  We do not know whether
  this diagram is complete for computable reducibility.\label{fig:ei}}
\end{figure}
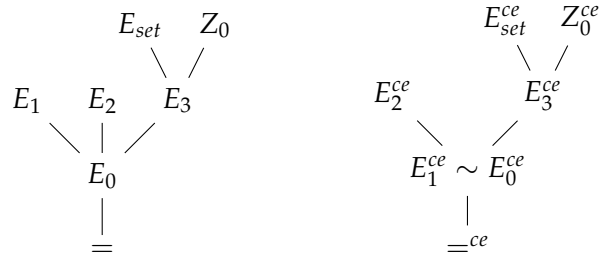

This result is summarized in Figure~\ref{fig:ei}.  We remark that the
diagram is complete for Borel reducibility, in the sense that any edge
not shown is known to correspond to a nonreduction.  For computable
reducibility, Theorem \ref{thm:e3} uses the arithmetic hierarchy to
show that the missing edges involving $E_3^{ce}$ remain nonreductions.
The other missing edges are more difficult, and among them, the only
one which we have settled appears further below as Theorem
\ref{thm:E0E1}. Surprisingly, this theorem shows that $E_1^{ce}\leq
E_0^{ce}$, which distinguishes the hierarchy for computable
reducibility from that for Borel reducibility.

\begin{thm}
  \label{thm:e3}
  $E_3^{ce}$ is not computably reducible to any of $E_0^{ce}$,
  $E_1^{ce}$ or $E_2^{ce}$.
\end{thm}

\begin{proof}
  We use a direct arithmetic complexity argument.  Observe that
  $E_0^{ce}$, $E_1^{ce}$, and $E_2^{ce}$ are all easily seen to be
  $\Sigma^0_3$.  On the other hand, we shall show that $E_3^{ce}$ is
  not $\Sigma^0_3$.  In fact, we shall show that $E_3^{ce}$ has one
  equivalence class which is $\Pi^0_3$ complete, namely, the set $F$
  consisting of indices $e$ for c.e.\ subsets of $\NN\times\NN$ such
  that every column of $W_e$ is finite.  (As a matter of fact,
  $E_3^{ce}$ is $\Pi^0_4$-complete, but we will not prove this here.)

  To begin, consider any set of the form $A=\set{x\mid(\forall
    n)(\exists N)(\forall k)\phi(x,n,N,k)}$, where $\phi$ is
  computable.  For each $x$, we shall write down a program $e_x$ which
  enumerates a c.e.\ subset of $\NN\times\NN$ such that $x\in A$ if
  and only if $e_x\in F$.  For each $n$, the program $e_x$ considers
  each $N$ in turn, and checks whether for all $k$, $\phi(x,n,N,k)$
  holds.  Whenever we find a counterexample $k$ for the current $N$,
  we enumerate one more element into the $n\th$ column of $W_{e_x}$.
  Observe that if $x\in A$, then for every $n$, we will eventually
  find the $N$ that works, and so the $n\th$ column of $W_{e_x}$ will
  be finite.  Hence in this case $e_x\in F$.  If $x\notin A$, then for
  some $n$, we will have to change our mind about $N$ infinitely
  often, and so the $n\th$ column of $W_{e_x}$ will be all of $\NN$.
  Hence in this case $e_x\notin F$.
\end{proof}

To prove that $E_1^{ce}\leq E_0^{ce}$, which surprised us, we will
build the necessary computable reduction.  This requires a detailed
construction.  To aid comprehension, we first present the basic module
for this construction, stated here as proposition
\ref{prop:sortof}. Note that this proposition does not provide a
reduction in the required sense, since the function $f$ depends on
both $i$ and $j$; nevertheless, understanding this proof will help the
reader follow the proof of the full Theorem \ref{thm:E0E1}.

\begin{prop}
  \label{prop:sortof}
  There exists a computable function $f$ such that, for every
  $i,j\in\omega$, we have
  $$ W_i~E_1~W_j \iff W_{f(i,j)}~E_0~W_{f(j,i)}.$$
\end{prop}

\begin{proof}
  On inputs $i$ and $j$, let the output $f(i,j)$ be (a G\"odel code
  for) the program which enumerates the set $W_{f(i,j)}$ we now
  describe.  For each element $\langle c,k\rangle$ and each stage
  $s+1$, enumerate $\langle c,k\rangle$ into $W_{f(i,j),s+1}$ if the
  following hold:
  \begin{itemize}
  \item $\langle c,k\rangle\in W_{i,s}\smallsetminus W_{j,s}$; and
  \item $(\forall n<k)~[\langle c,n\rangle\in W_{i,s}\iff \langle
    c,n\rangle\in W_{j,s}]$.
  \end{itemize}
  Additionally, if $\langle c,k\rangle\in W_{f(j,i),s}\cap W_{i,s}\cap
  W_{j,s}$, then enumerate $\langle c,k\rangle$ into $W_{f(i,j),s+1}$.
  This is the entire construction.

  Now fix $c$, and suppose that $W_i$ and $W_j$ agree on their $c$\th\
  columns: $(\forall k)[\langle c,k\rangle\in W_i\iff\langle
  c,k\rangle\in W_j]$.  Then, if $\langle c,k\rangle$ ever entered
  $W_{f(i,j)}$, it did so either because it had already entered
  $W_{f(j,i)}$, or else because $\langle c,k\rangle\in
  W_{i,s}\smallsetminus W_{j,s}$, in which case it must later have
  entered $W_j$ as well (by their agreement on the $c$\th\ column),
  and therefore was then enumerated into $W_{f(j,i)}$ as well.  Thus
  $W_{f(i,j)}$ and $W_{f(j,i)}$ agree with each other on their $c$\th\
  columns.

  On the other hand, suppose $W_i$ and $W_j$ disagree on their $c$\th\
  columns, and let $\langle c,k\rangle$ be the least point lying in
  just one of them.  If $\langle c,k\rangle\in W_i\smallsetminus W_j$,
  then by its minimality, it will eventually be enumerated into
  $W_{f(i,j),s+1}$ at some stage $s+1$.  It will never appear in
  $W_{f(j,i)}$.  Moreover, the only numbers $k'$ such that $\langle
  c,k'\rangle$ can be enumerated into either $W_{f(i,j)}$ or
  $W_{f(j,i)}$ at stages $>s+1$ are those which either sat in
  $W_{f(j,i),s+1}$ or $W_{f(i,j),s+1}$ already, or else have $k'<k$,
  because once $\langle c,k\rangle$ has appeared in $W_{i,s+1}$, no
  $\langle c,k'\rangle$ with $k'>k$ could ever again be the least
  difference between the $c$\th\ columns of $W_i$ and $W_j$.
  Therefore, $W_{f(i,j)}$ and $W_{f(j,i)}$ do differ on at least one
  element of their $c$\th\ columns, but do not differ by more than
  finitely many elements.  (In fact, each of these columns is finite.)

  A symmetric argument holds when the least difference lies in
  $(W_j-W_i)$.  Therefore, if $W_i~E_1~W_j$, then each of the finitely
  many columns on which they differ contains only finitely many
  differences between $W_{f(i,j)}$ and $W_{f(j,i)}$, and none of the
  other columns contain any differences at all between them.  Thus
  $W_{f(i,j)}~E_0~W_{f(j,i)}$.  Conversely, if $W_i$ and $W_j$ differ
  on infinitely many columns, then every one of those columns contains
  at least one difference between $W_{f(i,j)}$ and $W_{f(j,i)}$ as
  well, and so $W_{f(i,j)}$ and $W_{f(j,i)}$ are not $E_0$-related.
  Thus we have proven Proposition \ref{prop:sortof}.
\end{proof}

Adapting this basic module to cover all sets $W_e$ at once, rather
than just two particular sets $W_i$ and $W_j$, is challenging, but it
can be done.  In the following, we write $W_e^c = W_e\cap\set{\langle
  c,n\rangle\mid n\in\omega}$ for the $c$\th\ column of any c.e.\ set
$W_e$.  Also, for each $c$, we will need infinitely many infinite
subsets $C^{cj}$ of $\omega$ on which to execute our construction.  To
avoid confusion, we refer to these sets as \emph{slices} of $\omega$,
not as columns.

\begin{thm}
  \label{thm:E0E1}
  $E_1^{ce} \leq E_0^{ce}$. Hence, $E_1^{ce}$ and $E_0^{ce}$ are
  computably bireducible.
\end{thm}

\begin{proof}
%
%
%
  We define a computable function $g$ implicitly by building the c.e.\
  sets $W_{g(i)}$ uniformly for $i\in\omega$.  Here we give some
  intuition for the construction, before presenting it in full.  The
  slice $C^{cj}$ is dedicated to making sure that, if the set $W_j$
  differs on column $c$ from all the sets $W_i$ with $i<j$, then there
  should be at least one difference between $W_{g(j)}$ and each
  $W_{g(i)}$.  At each stage $s+1$ in the construction, we will
  consider all $c,j\in\omega$ and all $i<j$, and define two auxiliary
  elements: $m^{cij}_s$, which represents the least element of the
  $c$\th\ column $\omega^c$ which lies in the symmetric difference of
  $W_{i,s}$ and $W_{j,s}$, and $x^{cj}_s\in C^{cj}$, the element we
  are currently using to distinguish $W_{g(i)}$ from $W_{g(j)}$ on the
  slice $C^{cj}$.  The former will converge to the least element of
  $(W_i \mathbin{\triangle} W_j)\cap \omega^c$, or diverge if this set
  is empty.  Each time it changes, the current $x^{cj}_s$ enters
  $W_{g(i)}\cap W_{g(j)}$ and a new $x^{cj}_{s+1}$ is chosen and added
  to $W_{g(j)}$.  The same will happen occasionally on behalf of some
  $k>j$, as described below, but in the end, $W_{g(i)}$ and $W_{g(j)}$
  will agree on the slice $C^{cj}$ if and only if $W_i$ and $W_j$
  agree on the column $\omega^c$.

  We also must ensure that $(W_{g(i)}\mathbin{\triangle} W_{g(j)})\cap
  C^{cj}$ is finite, since $W_j$ and some $W_i$ might be equal on all
  other columns (hence might be $E_1$-related).  The $C^{cj}$ slice is
  not intended to differentiate any $W_i$ and $W_{i'}$ with $i'<i<j$;
  indeed, we want to have $W_{g(i)}\cap C^{cj}=W_{g(i')}\cap C^{cj}$
  for all such $i$ and $i'$, leaving it to $C^{ci}$ to differentiate
  between them if necessary.  Moreover, if any $i<j$ has
  $W_i^c=W_j^c$, then we will let $C^{ci}$ do the job, and
  $W_{g(j)}\cap C^{cj}$ will be equal to $W_{g(k)}\cap C^{cj}$ for
  \emph{every} $k$, so that we do not repeat the process of adding
  differences between the same sets.

  A set $W_k$ with $k>j$ might equal $W_j$ on the $c$\th\ column, or
  might equal some $W_i$ with $i<j$ there.  In the former case we
  would want to build $W_{g(k)}$ to be equal on $C^{cj}$ to
  $W_{g(j)}$, but in the latter case we would want it equal to
  $W_{g(i)}$ instead -- which could pose difficulties if $W_{g(j)}\neq
  W_{g(i)}$ on $C^{cj}$.  Our solution is to use the least difference
  (if any) between $W_i^c$ and $W_j^c$ to guess which is the case.  If
  this minimum element $m^{cij}=\lim_s m^{cij}_s$ lies in both $W_j$
  and $W_k$, for instance, then we can be sure that $W_k^c\neq W_i^c$,
  and if similar outcomes hold for each minimum (for every $i<j$),
  then we will make $W_{g(k)}$ look like $W_{g(j)}$ on $C^{cj}$.  On
  the other hand, if there is some minimum element which shows that
  $W_k^c\neq W_j^c$, then on this slice we will make $W_{g(k)}$ look
  like the sets $W_{g(i)}$ (which are the same on this slice for all
  $i<j$).  We do this for only finitely many $k>j$, namely those $\leq
  c$, to avoid having to add infinitely many elements to the sets
  $W_{g(j)}$ and $W_{g(i)}$, for that could wipe out the difference we
  would like to build between those sets.  For each single $k$, there
  are only finitely many values $j$ and columns $c$ for which this
  restriction will stop $W_k$ from being considered in the
  construction of $W_{g(k)}$ on $C^{cj}$; whenever either $j\geq k$ or
  $c>k$, $W_k$ will be used in that construction.  On the finitely
  many slices $C^{cj}$ in which it was not considered, $W_{g(k)}$ will
  have only a finite difference from any other set $W_{g(n)}$ anyway,
  which will have no effect on the question of whether
  $W_{g(k)}~E_0~W_{g(n)}$.

  For the slice $C^{cj}$, there are two basic outcomes possible.
  First, suppose some $i<j$ satisfies $W_i^c=W_j^c$.  Then there will
  be infinitely many stages $s$ at which $m^{cij}_{s+1}$ is either
  $\neq m^{cij}_s$ or undefined (if the symmetric difference is empty
  at that stage).  At every one of these stages, the current
  $x^{cj}_s$, which already sat in $W_{g(j)}$, will be added to every
  set $W_{g(k)}$ with $k\neq j$, and the newly chosen $x^{cj}_{s+1}$
  will be added to $W_{g(j)}$.  Since this happens infinitely often,
  we wind up with $C^{cj}\subseteq W_{g(k)}$ for every $k\in\omega$.

  On the other hand, suppose every $i<j$ has $W_i^c\neq W_j^c$.  Then
  every sequence $\langle m^{cij}_s\rangle_{s\in\omega}$ will converge
  to a limit $m^{cij}$, the minimum of $W_i^c\mathbin{\triangle}
  W_j^c$.  Once we reach a stage $s_0$ at which all
  $W_{i,s_0}\seqrestr{m^{cij}+1}=W_i\seqrestr{m^{cij}+1}$ and also
  $W_{j,s_0}\seqrestr{m^{cij}+1}=W_j\seqrestr{m^{cij}+1}$ for all $i$,
  the values $m^{cij}_s$ will never change again, and therefore will
  never cause $x^{cj}_{s_0}$ to enter any $W_{g(i)}$.  The indices $k$
  with $j<k\leq m$ (if there are any) may yet cause this element to
  enter the sets $W_{g(i)}$.  There are only finitely many such $k$,
  however, and each one causes this to happen at no more than one
  stage after $s_0$ (namely, the unique stage $s+1$, if one exists,
  such that $W_{j,s}^c$ and $W_{k,s}^c$ agree on the set of minima
  $\set{m^{cij}\mid i<j}$, but $W_{j,s+1}^c$ and $W_{k,s+1}^c$ fail to
  agree on that set).  Therefore, in this second outcome, there will
  exist a limit $x^{cj}=\lim_s x^{cj}_s$, which will lie in
  $W_{g(j)}$, will not lie in any $W_{g(i)}$ with $i<j$, and will lie
  in $W_{g(k)}$ (for $j<k\leq m$) if and only if $W_k$ and $W_j$
  contain exactly the same elements from the set of minima.  Moreover,
  this $x^{cj}$ will be the only element of $C^{cj}$ on which
  $W_{g(j)}$ differs from any set $W_{g(k)}$.  So in this case we have
  accomplished the goal of establishing a single difference on the
  slice $C^{cj}$ between $W_{g(i)}$ and $W_{g(j)}$ for each $i<j$,
  while not differentiating $W_{g(i)}$ from $W_{g(i')}$ on this slice
  for any $i'<i<j$.  The point of our treatment of the sets $W_{g(k)}$
  with $j<k\leq m$ was discussed above, and will appear below in Lemma
  \ref{lemma:goodcols}.  (The sets $W_{g(k)}$ with $k>m$ agree on
  $C^{cj}$ with all $W_{g(i)}$ for $i<j$, and hence differ by at most
  the element $x^{cj}$ from $W_{g(j)}$.)

  Now we give the formal construction.  At stage $0$, every element
  $m^{cij}_0$ is undefined, and $x^{cj}_0$ is the least element of the
  slice $C^{cj}$.  We also enumerate $x^{cj}_0$ into the set
  $W_{g(j),0}$ (which as yet contains nothing other than this
  element).

  At stage $s+1$, for each fixed $j$ and $c$ and for every $i<j$, we
  find the least element $m^{cij}_{s+1}$ (if any) of the symmetric
  difference $(W_{i,s}^c\mathbin{\triangle} W_{j,s}^c)$.  Suppose
  first that there exists either an $i<j$ such that this symmetric
  difference is empty, or an $i<j$ such that $m^{cij}_{s+1}\neq
  m^{cij}_s$.  Then we enumerate the $x^{cj}_s$ into every slice
  $W_{g(k),s+1}$, (we shall see that it was already in $W_{g(j),s}$),
  define $x^{cj}_{s+1}$ to be the next-smallest element of $C^{cj}$,
  and enumerate this $x^{cj}_{s+1}$ into $W_{g(j),s+1}$.

  On the other hand, suppose that for every $i<j$, $m^{cij}_{s+1}$ and
  $m^{cij}_s$ are defined and equal to each other.  In this case we
  consider each $k$ with $j<k\leq c$.  (If $j\geq c$, we do nothing.)
  For each such $k$ in turn, we ask whether the following fact holds
  at this stage $s+1$:
  \[(\forall i<j) [m^{cij}_s \in W_{k,s+1}\iff m^{cij}_s\in W_{j,s+1}]\;,
  \]
  and also whether the same fact held at the preceding stage $s$:
  \[(\forall i<j) [m^{cij}_s \in W_{k,s}\iff m^{cij}_s\in W_{j,s}]\;.
  \]
  If for some $k$ the fact held at the preceding stage but fails to
  hold now, then $x^{cj}_s$ already lies in all such slices
  $W_{g(k),s}$, and in this case we enumerate $x^{cj}_s$ into every
  $W_{g(i),s+1}$ for every $i$, define $x^{cj}_{s+1}$ to be the
  next-smallest element of $C^{cj}$, and enumerate this $x^{cj}_{s+1}$
  into $W_{g(j),s+1}$.  If there is no such $k$, we keep
  $x^{cj}_{s+1}=x^{cj}_s$.

  Finally, for every $k$ for which the fact holds at the current stage
  and failed to hold at the preceding stage, we enumerate into
  $W_{g(k),s+1}$ the element $x^{cj}_{s+1}$ chosen above.  This
  completes the construction.

  To show that the computable function $g$ defined by this
  construction actually constitutes a reduction from $E_1^{ce}$ to
  $E_0^{ce}$, we prove a series of lemmas. 

\begin{lem}
  \label{lemma:findiff}
  For every $c$, $j$, $m$ , and $n$, the sets $W_{g(m)}\cap C^{cj}$
  and $W_{g(n)}\cap C^{cj}$ differ on at most the element
  $x^{cj}=\lim_s x_s^{cj}$, if this limit exists.  If $W_i^c=W_j^c$
  for some $i<j$, then $x^{cj}$ does not exist, and the two sets above
  are equal.
\end{lem}

\begin{proof}
  At each stage $s$, on the slice $C^{cj}$, these two sets differ on
  at most $x^{cj}_s$, which enters $W_{g(j)}$ at one stage, might
  possibly enter certain sets $W_{g(k)}$ with $j<k\leq m$ at a
  subsequent stage, and then enters all other sets $W_{g(m)}$ when and
  if a new $x^{cj}_{s+1}$ is chosen.  No elements of $C^{cj}$ ever
  enter any of these sets except those chosen at some stage $t$ as
  $x^{cj}_t$.  So the lemma is clear.
\end{proof}

\begin{lem}
  \label{lemma:diffexists}
  For all triples $n<m<c$ with $W_m^c\neq W_n^c$ there is some $j\leq
  m$, the slice $C^{cj}$ has a limit element $x^{cj}\in
  W_{g(n)}\mathbin{\triangle} W_{g(m)}$.
\end{lem}

\begin{proof}
  If every $i<m$ satisfies $W_i^c\neq W_m^c$, then we simply take
  $j=m$, and the above analysis shows that $x^{cj}\in C^{cj}$ exists
  and lies in $W_{g(j)}$ (that is, in $W_{g(m)}$), but not in
  $W_{g(n)}$, since $n<j$.

  Otherwise, choose the least $m'<m$ with $W_{m'}^c=W_m^c$ and the
  least $n'\leq n$ with $W_{n'}^c=W_n^c$.  Let $i=\min(m',n')$ and
  $j=\max (m',n')$; thus $i<j$.  (With $W_{n'}^c=W_n^c\neq W^c_{m}=
  W_{m'}^c$, we know $i\neq j$.)
%
  Now no $i'<j$ satisfies $W_j^c=W_{i'}^c$, so there exists a limit
  element $x^{cj}$ which lies in $W_{g(j)}$ but not in $W_{g(i)}$.

  Suppose first that $W_m^c=W_j^c$.  Then we have $j\leq m< c$, and so
  eventually $W_m$ and $W_j$ agree on all minima $m^{ci'j}$, for all
  $i'<j$.  At every subsequent stage, $W_{g(m)}\cap C^{cj}$ will equal
  $W_{g(j)}\cap C^{cj}$, so in particular $x^{cj}\in W_{g(m)}$.  Now
  if $n\leq j$, then $n<j$ because $W_j^c=W_m^c\neq W_n^c$ and so
  automatically $x^{cj}$ will not lie in $W_{g(n)}$.  On the other
  hand, if $j<n$, then with $W_n^c=W_i^c$ and $i<j$, $m^{cij}$ must
  lie in $W_n^c\mathbin{\triangle} W_j^c$.  Once $m^{cij}$ has
  appeared in one of these sets, $W_{g(n)}$ will fail to contain
  $x^{cj}_s$ at all subsequent stages $s$, and therefore will fail to
  contain the limit $x^{cj}$.  Thus $x^{cj}$ lies in
  $W_{g(m)}\smallsetminus W_{g(n)}$, giving the difference we desired.

  The exact same argument holds if $W_n^c=W_j^c$, only with $m$ and
  $n$ reversed.  So we have proven the lemma.
\end{proof}

We can now establish one direction of Theorem~\ref{thm:E0E1}, namely
that if $W_m\not\mathrel{E}_1W_m$ then
$W_{g(m)}\not\mathrel{E}_0W_{g(n)}$.  Indeed, if $W_m$ and $W_n$
differ on infinitely many columns $c$, then they differ on infinitely
many columns $c$ with $c>\max (m,n)$.  For each such $c$,
Lemma~\ref{lemma:diffexists} gives a $j$ and an element $x^{cj}\in
C^{cj}$ on which $W_{g(m)}$ and $W_{g(n)}$ differ, as desired.  It
remains to prove the converse, namely that if $W_m\mathrel{E}_1W_n$
then $W_{g(m)}\mathrel{E}_0W_{g(n)}$.  To begin with, we consider the
columns $c$ on which they agree.

\begin{lem}
  \label{lemma:goodcols}
  Fix $m<n$ and $c\geq n$.  If $W_m^c=W_n^c$, then for all
  $j\in\omega$ we have $W_{g(m)}\cap C^{cj} = W_{g(n)}\cap C^{cj}$.
\end{lem}

\begin{proof}
  Of course, if $W^c_j=W^c_i$ for some $i<j$, then every $W_{g(k)}$
  contains all of $C^{cj}$, and we are done.  So assume that there
  exists no such $i<j$, and that therefore the limit $x^{cj}$ is
  defined.  We consider all possible values of $j$ relative to $m$ and
  $n$.

  If $j=n$, then there is an $i<j$ (namely $i=m$) with $W_i^c=W_j^c$,
  and so $W_{g(n)}\cap C^{cj}= W_{g(m)}\cap C^{cj}=C^{cj}$.

  If $j=m$, then $W_n^c$ agrees with $W_j^c$ on all minima $m^{cij}$
  with $i<j$, and so $W_{g(j)}$ and $W_{g(n)}$ both contain $x^{cj}$,
  and therefore are equal on $C^{cj}$.  (This uses the fact that
  $n\leq c$.)

  If $j<m$, then $W_m^c$ and $W_n^c$ either both agree with $W_j^c$ on
  all minima, or both disagree with it on some particular minimum.  In
  the former case, they both contain $x^{cj}$, while in the latter
  case they both omit it.  Either way they both are equal, by Lemma
  \ref{lemma:findiff}.

  If $j>n$, then automatically $W_{g(m)}\cap C^{cj}= W_{g(n)}\cap
  C^{cj}$, because for all $i<j$, $W_{g(i)}$ contains exactly those
  elements of $C^{cj}$ which are $<x^{cj}$.

  Finally, suppose $m<j<n$.  If $W_j^c\neq W_i^c$ for every $i<j$,
  then in particular $W_j^c\neq W_m^c=W_n^c$.  So $x^{cj}$ fails to
  lie in $W_m^c$ (since $m<j$) and also fails to lie in $W_n^c$
  (because they differ on the minimum $m^{cmj}$).  By Lemma
  \ref{lemma:findiff}, therefore, $W_{g(m)}$ and $W_{g(n)}$ agree on
  $C^{cj}$.
\end{proof}

\begin{lem}
  \label{lemma:goodcols2}
  Fix $m<n$, and assume $W_m^c=W_n^c$.  If $W_{g(m)}\cap C^{cj} \neq
  W_{g(n)}\cap C^{cj}$, then $c<n$ and $j\leq n$, and the symmetric
  difference $(W_{g(m)}\cap C^{cj}) \mathbin{\triangle} (W_{g(n)}\cap
  C^{cj})$ is finite.
\end{lem}

\begin{proof}
  That $c<n$ follows from Lemma \ref{lemma:goodcols}.  Moreover, we
  know that $W_{g(i)}\cap C^{cj}=W_{g(i')}\cap C^{cj}$ for every
  $i<i'<j$, and if $n<j$, then this applies to $m$ and $n$.  Finally,
  Lemma \ref{lemma:findiff} shows that the symmetric difference of the
  two sets contains at most one element.
\end{proof}

\begin{lem}
  \label{lem:finite1}
  Suppose $W_m~E_1~W_n$, and let the set $C$ be the union of all those
  $C^{cj}$ with $j\in\omega$ and $W_m^c=W_n^c$.  Then $(W_{g(m)}\cap
  C)$ and $(W_{g(n)}\cap C)$ differ by at most finitely many elements.
\end{lem}

\begin{proof}
  By Lemma \ref{lemma:goodcols2}, the difference is contained within
  finitely many slices $C^{cj}$, and that difference is finite (in
  fact, at most a single element) on each of those finitely many
  $C^{cj}$.
\end{proof}

So, for $m<n$ with $W_m~E_1~W_n$, it remains to consider $W_{g(m)}$
and $W_{g(n)}$ on those slices $C^{cj}$ with $W_m^c\neq W_n^c$.  There
are only finitely many such $c$, but there are infinitely many
corresponding $j$.  However, all but finitely many of these $j$
satisfy $m<n<j$, and for all those $j$, we know that $W_{g(m)}\cap
C^{cj} =W_{g(n)}\cap C^{cj}$.  On the finitely many remaining sets
$C^{cj}$, there may be a difference, but only a finite difference, by
Lemma \ref{lemma:findiff}. Thus, whenever $W_m~E_1~W_n$, we must have
$W_{g(m)}~E_0~W_{g(n)}$. This completes the proof of Theorem
\ref{thm:E0E1}.
\end{proof}

Using Proposition \ref{prop:Borel} we also conclude that $E_1^{ce}\leq
E_2^{ce}$ and $E_1^{ce}\leq E_3^{ce}$.  The remaining questions in
establishing a version of Figure \ref{fig:ei} for computable
reducibility are whether $E_2^{ce}\leq E_1^{ce}$ and whether either
$E_{set}^{ce}$ or $Z_0^{ce}$ is $\leq E_3^{ce}$.

\section{Below equality}

In this section, we add to the collection of known
equivalence relations on c.e.\ sets which either lie properly below
$=^{ce}$, or else are incomparable with $=^{ce}$, in the computable
reducibility hierarchy.  This is a departure from the Borel theory,
where Silver's theorem implies that $=$ is continuously reducible to
every nontrivial Borel equivalence relation.  On the other hand, the
situation is not entirely unfamiliar, being similar to that for other
very weak reducibility notions (see for instance \cite{buss} for the
case of PTIME reducibility).  Indeed, all computably enumerable
equivalence relations $E$ on $\omega$ are computably reducible
to $=^{ce}$, just by letting $W_{f(e)}=\{i~:~\la e,i\ra\in E\}$,
and so the work done in \cite{sorbi} and \cite{gerdes} all takes
place within this realm.

We begin with the following examples, each of which is a simple (but
not c.e.) equivalence relation.

\begin{defn}\
  \label{def:minmax}
  \begin{itemize}
  \item Let $e\mathrel{E}_\smin^{ce} e'$ if and only if
    $\min(W_e)=\min(W_{e'})$ or $W_e=W_e'=\emptyset$.
  \item Let $e\mathrel{E}_\smax^{ce} e'$ if and only if either
    $W_e,W_{e'}\neq\emptyset$ and $\max(W_e)=\max(W_{e'})$, or
    $W_e=W_{e'}=\emptyset$, or $\abs{W_e}=\abs{W_{e'}}=\aleph_0$.
  \end{itemize}
\end{defn}

It is easy to see that these relations are computably reducible to
$=^{ce}$.  Indeed, to show $E_\smin^{ce}\leq\mathord{=}^{ce}$, given a
program $e$ we saturate $W_e$ upwards by letting $f(e)$ enumerate the
set $\set{n\in\NN\mid\exists l\in W_e(l\leq n)}$.  Similarly, to show
$E_\smax^{ce}\leq\mathord{=}^{ce}$, we saturate downwards with the
program $f(e)$ that enumerates the set $\set{n\in\NN\mid\exists l\in
  W_e(l\geq n)}$. In both cases, these functions are computable
selectors, in the sense that $W_{f(e)}$ is $E_\smin^{ce}$ or
$E_\smax^{ce}$ equivalent to $W_e$ and constant on these classes.

\begin{rem}
  It is worth mentioning that $E_\smin^{ce}$ and $E_\smax^{ce}$ each
  admit another simple description.  Namely, $E_\smin^{ce}$ is
  computably bireducible with the relation $e\mathrel{E}_\sgcd^{ce}e'$
  if and only if $\gcd(W_e)=\gcd(W_{e'})$.  (Here $\gcd(S)$ is the
  greatest $n$ which divides every element of $S$, with
  $\gcd(\emptyset)=\infty$.)  Indeed, to show $E_\smin^{ce}\leq
  E_\sgcd^{ce}$, we let $W_{f(e)}=\set{n!\mid n\in W_e}$; while to
  show $E_\sgcd^{ce}\leq E_\smin^{ce}$, we let
  $W_{f(e)}=\set{d<\infty\mid\exists s(d=\gcd(W_{e,s}))}$.  Similarly,
  $E_\smax^{ce}$ is bireducible with the relation
  $e\mathrel{E}_\slcm^{ce}e'$ if and only if $\lcm(W_e)=\lcm(W_{e'})$.
\end{rem}

The next result gives our first example of a violation of Silver's
theorem in the computable context.

\begin{thm}
  \label{thm:maxtomin}
  $E_\smax^{ce}$ is not computably reducible to $E_\smin^{ce}$.
  Consequently, $E_\smin^{ce}$ lies properly below $=^{ce}$.
\end{thm}

\begin{proof}
  This holds for the simple reason that $E_\smin^{ce}$ is
  $\Delta^0_2$, while $E_\smax^{ce}$ is $\Pi^0_2$ complete.  To see
  that $E_\smin^{ce}$ is $\Delta^0_2$, observe that $W_e$ and $W_{e'}$
  have the same minimum if and only if for every $n$ in $W_e$ there
  exists an $m\leq n$ in $W_{e'}$ and vice versa; and also if and only
  if there exists $n$ in $W_e$ and $W_{e'}$ such that every $m\leq n$
  is in neither $W_e$ nor $W_{e'}$ (or both are empty).  To see that
  $E_\smax^{ce}$ is $\Pi^0_2$ complete, note that by
  \cite[Theorem~IV.3.2]{soare}, even the $E_\smax$ class
  $\mathrm{INF}=\set{e:\abs{W_e}=\aleph_0}$ is $\Pi^0_2$ complete.
\end{proof}

In fact, $E_\smax^{ce}$ and $E_\smin^{ce}$ are \emph{incomparable} up
to computable reducibility, and hence both lie properly below
$=^{ce}$.  However, the proof that $E_\smin^{ce}$ is not computably
reducible to $E_\smax^{ce}$ is slightly more difficult.  For this, we
require the monotonicity lemma, a key result which will be used a
number of times over the next few sections.  Before stating it, we
need the following terminology.

\begin{defn}\
  \begin{itemize}
  \item We say that $f\from\NN\into\NN$ is \emph{well-defined on c.e.\
      sets} if $W_e=W_{e'}$ implies $W_{f(e)}=W_{f(e')}$ (in other
    words, $f$ is a homomorphism from $=^{ce}$ to $=^{ce}$).
  \item We say that such a function $f$ is \emph{monotone} if
    $W_e\subset W_{e'}$ implies $W_{f(e)}\subset W_{f(e')}$.
  \end{itemize}
\end{defn}

\begin{lem}[Monotonicity lemma]
  \label{lem:monotonicity}
  If $f\from\NN\into\NN$ is computable and well-defined on c.e.\ sets,
  then $f$ is monotone.
\end{lem}

\begin{proof}
  Fix $e$, $e'$, and $x$ such that $W_e\subset W_{e'}$, and $x\in
  W_{f(e)}$; we must show that $x\in W_{f(e')}$ as well.  We shall use
  the recursion theorem to design an auxiliary program $p$ which knows
  its own index, and hence the index of $f(p)$.  The program $p$
  simulates both $f(p)$ and $e$, and at first $p$ behaves just like
  $W_e$.  Of course, if $p$ were to continue in this manner forever,
  then we would have $W_p=W_e$ and since $f$ is well-defined on c.e.\
  sets, we would have $W_{f(p)}=W_{f(e)}$.  It follows that there is
  some stage by which $x$ appears in $W_{f(p)}$.  Once this occurs,
  $p$ begins to simulate $e'$ and mimic \emph{its} behavior instead of
  that of $e$.  This does not contradict the earlier behavior of $p$,
  since $W_e\subset W_{e'}$.  Thus in the end, we will have
  $W_p=W_{e'}$, and hence $W_{f(p)}=W_{f(e')}$, since $f$ is
  well-defined on c.e.\ sets.  But we also arranged that $x\in
  W_{f(p)}$ and therefore $x\in W_{f(e')}$, as desired.
\end{proof}

We can now complete the proof that $E_\smin^{ce}$ and $E_\smax^{ce}$
are computably incomparable.

\begin{thm}
  \label{thm:mintomax}
  $E_\smin^{ce}$ is not computably reducible to $E_{\smax}^{ce}$.
  Consequently, $E_\smax^{ce}$ lies properly below $=^{ce}$.
\end{thm}

\begin{proof}
  Suppose to the contrary that $f$ is a reduction from $E_\smin^{ce}$
  to $E_\smax^{ce}$.  We first claim that we can assume, without loss
  of generality, that $f$ is well-defined on c.e.\ sets.  Indeed let
  $g$ be the reduction from $E_\smax^{ce}$ to $=^{ce}$ given just
  below Definition~\ref{def:minmax}.  Then clearly $g\circ f$ is
  well-defined on c.e.\ sets.  Moreover since $g$ is a
  \emph{selector}, that is, $g(e)\mathrel{E_\smax^{ce}}e$, we have
  that $g\circ f$ is again a reduction from from $E_\smin^{ce}$ to
  $E_\smax^{ce}$.  Hence, we may replace $f$ with $g\circ f$ to
  establish the claim.

  Now, for each $n$, let $e_n$ be a program enumerating $[n,\infty)$.
  Then the sets $W_{e_n}$ form a monotone decreasing sequence of sets.
  By the claim, the monotonicity lemma implies that $W_{f(e_n)}$ is
  also a monotone decreasing sequence of sets.  Moreover, since the
  $\min(W_{e_n})$ are all distinct and $f$ is a reduction, we must
  have that $\max(W_{e_n})$ are all distinct.  It follows that
  $\max(W_{e_n})$ is a strictly decreasing sequence of natural
  numbers, which is a contradiction.
\end{proof}

At the end of the section, we will give a broad generalization of this
argument.  Before doing so, we will use these ideas to find an
equivalence relation on c.e.\ sets which is incomparable with
$=^{ce}$.

\begin{defn}
  Let $E_\smed^{ce}$ denote the equivalence relation on c.e.\ sets
  defined by $e\mathrel{E}_\smed^{ce} e'$ if and only if the sets
  $W_e$ and $W_{e'}$ are both finite and have the same median (or are
  both empty or both infinite).
\end{defn}

\begin{prop}
  \label{prop:med1}
  $E_\smed^{ce}$ lies between $E_\smax^{ce}$ and $E_0^{ce}$ in the
  reducibility hierarchy.
\end{prop}

\begin{proof}
  First, $E_\smax^{ce} \leq E_\smed^{ce}$ by the function $f$ which
  saturates downwards, that is, such that $W_{f(e)}=\set{n\mid\exists
    m\in W_e(n\leq m) }$.  Next, $E_\smed^{ce}$ is reducible to
  $E_0^{ce}$ via the function $e\mapsto f(e)$ defined as follows. The
  program $f(e)$ simulates $e$, and at each stage of simulation
  computes the median $r_s$ of the stage $s$ approximation $W_{e,s}$.
  As long as $r_s$ does not change, $f(e)$ will enumerate multiples of
  $r_s$ into $W_{f(e)}$.  Whenever $r_s$ does change, $f(e)$ fills in
  everything up to its current maximum and starting there enumerates
  multiples of the new $r_s$.

  Now, if $W_e$ and $W_{e'}$ both have median $r$, then by some stage
  the medians of $W_{e,s}$ and $W_{e',s}$ will have both stabilized at
  $r$.  Hence both $W_{f(e)}$ and $W_{f(e')}$ will both eventually
  contain just the multiples of $r$.  If $W_e$ and $W_{e'}$ are both
  empty or both infinite, then $W_{f(e)}$ and $W_{f(e')}$ will both be
  empty or all of $\NN$, respectively.  Finally, if $W_e$ and $W_{e'}$
  have distinct medians, then $W_e$ and $W_{e'}$ will disagree on an
  infinite set.
\end{proof}

\begin{thm}
  \label{thm:med2}
  $E_\smed^{ce}$ is incomparable with both $E_\smin^{ce}$ and $=^{ce}$
  with respect to computable reducibility.
\end{thm}

\begin{proof}
  It suffices to show that $E_\smin^{ce}$ is not reducible to
  $E_\smed^{ce}$ and that $E_\smed^{ce}$ is not reducible to $=^{ce}$.
  Suppose first that $f$ is a reduction from $E_\smed^{ce}$ to
  $=^{ce}$.  Then $f$ is well-defined on c.e.\ sets, and therefore it
  is monotone.  Consider the three sets $W_{e_1}=\{1\}$,
  $W_{e_2}=\{1,2\}$, and $W_{e_3}=\{0,1,2\}$.  Since $W_{e_1}\subset
  W_{e_2}$ and $\med(W_{e_1})\neq\med(W_{e_2})$, we have
  $W_{f(e_1)}\subsetneq W_{f(e_2)}$.  Since $W_{e_2}\subset W_{e_3}$
  and $\med(W_{e_2})\neq\med(W_{e_3})$, we have $W_{f(e_2)}\subsetneq
  W_{f(e_3)}$.  It follows that $W_{f(e_1)}\neq W_{f(e_3)}$.  But this
  contradicts that $f$ is a reduction, since
  $\med(W_{e_1})=\med(W_{e_3})$.

  Next suppose that $f$ is a reduction from $E_\smin^{ce}$ to
  $E_\smed^{ce}$, which means that $W_e E_\smin W_{e'}\iff W_{f(e)}
  E_\smed W_{f(e')}$.  Fix any index $e_0$ for which $W_{e_0}$ is
  nonempty and $W_{f(e_0)}$ has finite median $r_0$.  Let $N=2r_0+2$
  and choose a finite sequence of programs $e_1, e_2, \ldots, e_N$,
  for which
  $\min(W_{e_0})<\min(W_{e_1})<\min(W_{e_2})<\cdots<\min(W_{e_N})$ and
  each $r_i=\med(W_{f(e_i)})$ exists and is finite. Such a sequence
  exists because there are infinitely many different possible minimums
  for $W_{e_i}$ and at most two of these minimums leads to
  $W_{f(e_i)}$ being empty or infinite; all the rest have finite
  medians.  Note also that the $r_i$ are distinct. Consider now the
  program $e$ that begins by enumerating $\min(W_{e_N})$ into
  $W_e$. Since so far this set has the same minimum as $W_{e_N}$, it
  must eventually happen that there is a stage $s_N$ at which
  $W_{f(e),s_N}$ has median $r_N$. At such a stage, let program $e$
  now enumerate $\min(W_{e_{N-1}})$ into $W_e$. Because $W_e$ at this
  stage has the same minimum as $W_{e_{N-1}}$, if we do not add
  anything more to $W_e$ then there must be a stage $s_{N-1}$ at which
  the median of $W_{f(e),s_{N-1}}$ becomes $r_{N-1}$. And when this
  occurs, let program $e$ enumerate $\min(W_{e_{N-2}})$ into $W_e$,
  and so on. After $N$ iterations of this process, we have a set $W_e$
  with the same minimum as $W_{e_1}$, and at some eventual stage $s_1$
  the median of $W_{f(e),s_1}$ becomes $r_1$. When this occurs,
  finally, we enumerate $\min(W_{e_0})$ into $W_e$, thereby ensuring
  that the minimum of $W_e$ is the same as that of $W_{e_0}$, and so
  the median of $W_{f(e)}$ must now eventually become $r_0$. But a
  careful examination of our construction will reveal that $W_{f(e)}$
  has at least $N$ elements, that is, at least $2r_0+2$ many, since at
  least one new element was added at each step of the process. But if
  $W_{f(e)}$ has this many elements, then it is impossible for it to
  have median $r_0$, which is a contradiction. Therefore, no such
  reduction $f$ exists.
\end{proof}

The relationships expressed by the last two results are summarized in
Figure~\ref{fig:med}.

\begin{figure}[ht]
\begin{tikzpicture}
  \node at (0,0) (eq) {$=^{ce}$};
  \node at (-1,-1) (min) {$E_\smin^{ce}$} edge (eq);
  \node at (1,-1) (max) {$E_\smax^{ce}$} edge (eq);
  \node at (2,0) (med) {$E_\smed^{ce}$} edge (max);
  \node at (1,1) (e0) {$E_0^{ce}$} edge (eq) edge (med);
\end{tikzpicture}
\caption{Diagram of reducibility among the relations in this section
  thus far.  It follows from Theorems~\ref{thm:maxtomin},
  \ref{thm:mintomax}, and~\ref{thm:med2} that this diagram is
  complete, in the sense that any edges not shown represent
  non-reducibilities.\label{fig:med}}
\end{figure}
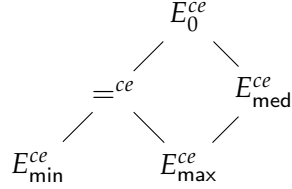

It is natural to generalize the examples of $E_\smin$ and $E_\smax$ to
try to find a large family of simple incomparable relations.  There
are many possibilities for doing so, and in the remainder of this
section we shall consider just one: a generalization to arbitrary
computable linear orders.  This will enable us to find numerous
equivalence relations which are incomparable and lie below $=^{ce}$,
and therefore strengthen our denial of Silver's theorem for computable
reducibility.

In what follows, if $L$ is a linear ordering with order relation $<_L$
and $W\subset L$, then we let $\cut_L(W)$ denote the Dedekind cut
determined by $W$, that is,
\[\cut_L(W)=\set{l\in L\mid\exists w\in W(l<_Lw)}\;.
\]
Evidently, if $L$ is a computable linear ordering with domain $\NN$,
and $W$ is a c.e.\ subset of $\NN$, then $\cut_L(W)$ is c.e.\ as well.

\begin{defn}\
  \begin{itemize}
  \item For $L$ a computable linear ordering, let $E_L$ denote the
    \emph{same cut} equivalence relation defined by $e\mathrel{E_L}e'$
    if and only if $\cut_L(W_e)=\cut_L(W_{e'})$.
  \item Similarly, we define $H_L$ to be the \emph{same hull}
    equivalence relation defined by $e\mathrel H_Le'$ if and only if
    the convex hull of $W_e$ in $L$ is the same as the convex hull of
    $W_{e'}$ in $L$.
  \end{itemize}
\end{defn}

For any computable $L$, the relations $E_L$ and $H_L$ are both
computably reducible to $=^{ce}$ (by mapping $W_e$ to the cut or the
hull that it determines, respectively).  Moreover, both $E_L$ and
$E_{L^*}$ are computably reducible to $H_L$, where $L^*$ denotes the
reverse of $L$.  To see this, note that the cut map again defines a
reduction $E_L\leq H_L$, and of course $H_L$ is bireducible with
$H_{L^*}$, since in fact $H_L$ and $H_{L^*}$ are the same relation.

We have already seen several of the $E_L$ in another context.  For
instance, the relation $E_\smax$ can be identified as $E_\omega$,
where $\omega$ denotes the usual ordering on $\NN$.  Similarly,
$E_\smin$ can be identified with $E_{\omega^*}$.  Finally, $=^{ce}$ is
computably bireducible with $E_\QQ$; to see that $=^{ce}\leq E_\QQ$
consider the map which sends a c.e.\ set $W$ to the cut in $\QQ$
corresponding to the real number $\sum_{n\in W}1/3^{n+1}$.

We shall use the following notation.  For $L$ a computable linear
order, let $\overline{L}$ denote the set of c.e.\ cuts in $L$.  We
shall say that $\overline{L_1}$ is \emph{computably embeddable} into
$\overline{L_2}$, written $\overline{L_1}\injto_c\overline{L_2}$, if
there exists a computable function $\alpha\from\NN\into\NN$ such that
for all programs $e,e'$, we have
\[\cut_{L_1}(W_e)<\cut_{L_1}(W_{e'})
  \iff\cut_{L_2}(W_{\alpha(e)})<\cut_{L_2}(W_{\alpha(e')})\;.
\]
The next result characterizes the structure of the $E_L$ relations
with respect to computable reducibility.

\begin{thm}
  \label{thm:lo}
  Let $L_1$ and $L_2$ be computable linear orders.  Then $E_{L_1}\leq
  E_{L_2}$ if and only if $\overline{L_1}\injto_c\overline{L_2}$.
\end{thm}

\begin{proof}
  This is another application of the monotonicity lemma.  Suppose
  first that $f$ is a computable reduction from $E_{L_1}$ to
  $E_{L_2}$.  We can suppose without loss of generality that $f$ is
  well-defined on c.e.\ sets.  (Indeed, simply post-compose $f$ with
  the map that sends $p$ to a program enumerating $\cut_{L_2}(W_p)$.)
  It follows that $f$ is $\subset$-preserving and hence preserves the
  ordering on cuts.  Moreover, since $f$ is a reduction, it must be
  injective on cuts.  Hence it is an embedding of $\overline{L_1}$
  into $\overline{L_2}$.

  Next suppose that $\alpha\from\overline{L_1}\injto_c\overline{L_2}$.
  Then we simply define $W_{f(e)}=\cut_{L_2}(W_{\alpha(e)})$, so that
  \begin{align*}
    e \mathrel{E}_{L_1} e'&\iff \cut_{L_1}(W_e)=\cut_{L_1}(W_{e'})\\
    &\iff \cut_{L_2}(W_{\alpha(e)})=\cut_{L_2}(W_{\alpha(e')})\\
    &\iff W_{f(e)}=W_{f(e')}\\
    &\iff f(e) \mathrel{E}_{L_2} f(e')\;.
  \end{align*}
  (The last equivalence holds because $W_{f(e)}$ and $W_{f(e')}$ are
  both cuts.)  Hence, $f$ is a reduction from $E_{L_1}$ to $E_{L_2}$.
\end{proof}

\begin{rem}
  This result can also be generalized without much effort to the
  ``same downwards closure'' equivalence relation on c.e.\ subsets of
  computable \emph{partial} orders.
\end{rem}

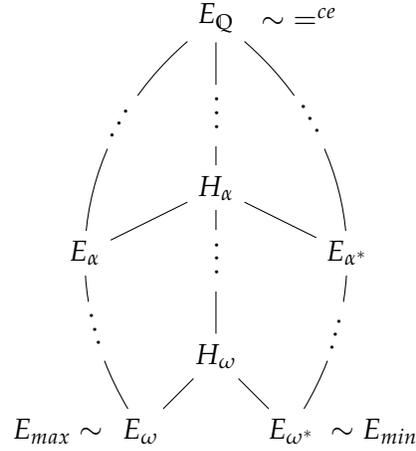
\begin{figure}[ht]
\begin{tikzpicture}[bend angle=40]
  \node at (0,0) (hullom) {$H_\omega$};
  \node at (-1,-1) (cutom) {$E_\omega$} edge (hullom);
  \node at (-2.1,-1) {$E_{max}\sim$};
  \node at (1,-1) (cutoms) {$E_{\omega^*}$} edge (hullom);
  \node at (2.1,-1) {$\sim E_{min}$};
  \node at (0,4.5) (q) {$E_\QQ$}
    edge[bend right]
      node[fill=white,sloped,pos=.25] {$\cdots$}
      node[fill=white,pos=.6] (cutal) {$E_\alpha$}
      node[fill=white,sloped,pos=.8] {$\cdots$} (cutom)
    edge[bend left]
      node[fill=white,sloped,pos=.25] {$\cdots$}
      node[fill=white,pos=.6] (cutals) {$E_{\alpha^*}$}
      node[fill=white,sloped,pos=.8] {$\cdots$} (cutoms)
    edge
      node[fill=white,sloped,pos=.25] {$\cdots$}
      node[fill=white] (hullal) {$H_\alpha$}
      node[fill=white,sloped,pos=.75] {$\cdots$} (hullom);
  \draw[-] (cutal) -- (hullal);
  \draw[-] (cutals) -- (hullal);
  \node at (1.1,4.5) {$\sim\mathord{=}^{ce}$};
\end{tikzpicture}
\caption{Diagram showing the cut and hull relations for computable
  ordinals $\alpha$ and their reverse orderings $\alpha^*$.  In 2017
  this will be the first diagram to land on Gliese
  581g.\label{fig:rocket}}
\end{figure}

Figure~\ref{fig:rocket} elaborates on Figure~\ref{fig:med} by showing
the relationships between some sample $E_L$ which hold thanks to
Theorem~\ref{thm:lo}. Note that the reductions are strict as one moves
to larger ordinals, because the cuts of a larger ordinal cannot map
into the cuts of a smaller ordinal, and no infinite well-order can map
into the cuts of its reverse. We end this section by mentioning a
couple more easy consequences of this result.

\begin{cor}
  As long as $\overline{L}\not\injto_c\overline{L^*}$, we have that
  $E_L$ and $E_{L^*}$ lie properly below $H_L$.
\end{cor}

\begin{cor}
  There exist infinite chains and arbitrarily large finite antichains
  of equivalence relations on c.e.\ sets which lie below $=^{ce}$.
\end{cor}

\begin{proof}[Sketch of proof.]
  For instance, to construct an antichain of size three, consider
  computable copies of $L_1=\omega+\omega+\omega^*$,
  $L_2=\omega+\omega^*+\omega$, and $L_3=\omega^*+\omega+\omega$.
  (For convenience, assume that all of the cuts are computable.)  Then
  it is easy to check that for $i\neq j$ we have
  $\overline{L_i}\not\injto\overline{L_j}$.
\end{proof}

\section{Enumerable equivalence relations}

A great portion of Borel equivalence relation theory focuses on the
countable Borel equivalence relations, that is, the Borel equivalence
relations with every class countable.  The foundation of this theory
is the Lusin/Novikov theorem from descriptive set theory which states
that every countable Borel equivalence relation can be enumerated in a
Borel fashion.  In other words, if $E$ is a countable Borel
equivalence relation on $X$, then there exists a Borel function
$f\from X\into X^\NN$ such that for all $x$, $f(x)$ enumerates
$[x]_E$.  Using this key fact, an argument of Feldman/Moore implies
that any countable Borel equivalence relation can be realized as the
orbit equivalence relation induced by a Borel action of a countable
group.  (See Theorem~7.1.2 and Theorem~7.1.4 of \cite{gao} for a
discussion of these results.)  In this section, we begin to develop a
computable analogue of countable Borel equivalence relations.

We begin by introducing an analogue of the Lusin/Novikov property for
equivalence relations on c.e.\ sets.

\begin{defn}
  Let $E$ be an equivalence relation on c.e.\ sets.  We say that
  $E^{ce}$ is \emph{enumerable in the indices} if there exists a
  computable function $\alpha\from\NN\times\NN\into\NN$ such that
  $e\mathrel{E}^{ce}e'$ if and only if there exists $n\in\NN$ such that
  $W_{\alpha(n,e)}=W_{e'}$.
\end{defn}

For example, $E_0^{ce}$ has this property.  To see this, let $s_n$
denote the $n\th$ element of some computable enumeration of
$2^{<\NN}$, and fix a function $\alpha$ such that the program
$\alpha(n,e)$ enumerates $s_n\concat
(W_e\smallsetminus\abs{s_n})$. Thus, as $n$ varies, the sets
$W_{\alpha(e,n)}$ enumerates all finite modifications of $W_e$, and so
$\alpha$ witnesses that $E_0^{ce}$ is enumerable in the indices.

\begin{prop}
  If $E^{ce}$ is enumerable in the indices then $E^{ce}\leq
  E_\sset^{ce}$.\qed
\end{prop}

Indeed, simply send a program $e$ to a program for a subset of
$\NN\times\NN$ which acts like $\alpha(n,e)$ on the $n\th$ column.  Of
course $E_\sset^{ce}$ is not itself enumerable, since enumerable
relations are easily seen to be $\Sigma^0_3$, whereas it follows from
Theorem~\ref{thm:e3} that $E_\sset^{ce}$ is $\Pi^0_3$ complete.

We next consider the important special case of equivalence relations
on c.e.\ sets which are induced, in an appropriate sense, by a
computable action of a computable group.

\begin{defn}
  Suppose that $\Gamma$ is a computable group acting on the c.e.\
  subsets of $\NN$.  We say that the action is \emph{computable in the
    indices} if there exists a computable function
  $\alpha\from\NN\times\NN\into\NN$ such that
  $W_{\alpha(\gamma,e)}=\gamma W_e$.
\end{defn}

For example, if $\Gamma$ is a computable group then the left
translation action of $\Gamma$ on the set $\mathcal P(\Gamma)^{ce}$ of
c.e.\ subsets of $\Gamma$ is computable in the indices.  We shall use
the following notation: if the group $\Gamma$ acts on $CE$, which we
write $\Gamma\actson CE$, then the induced orbit equivalence relation
is defined by $e\mathrel{E}_\Gamma^{ce} e'$ if and only if
$\exists\gamma\in\Gamma$ such that $W_{e'}=\gamma W_e$.

One would like to prove an analogue of the Feldman/Moore theorem which
would say that every enumerable relation is in fact the orbit relation
of some action which is computable in the indices.  Unfortunately,
this is not the case.  To get an idea of the difficulties involved,
consider the natural action giving rise to $E_0$, namely, the bitwise
addition action of $2^{<\NN}$ on $\mathcal P(\NN)$.  This action is
highly effective in many natural senses, but not in the sense of this
paper: given $e$ and $s\in2^{<\NN}$, it is not clear how to computably
produce a program which enumerates the bitwise sum
$s\stackrel{.}{+}W_e$.  Indeed, the next result strongly negates the
possibility that there is an analogue of Feldman/Moore in this
context.

\begin{thm}
  Let $E$ be an equivalence relation on c.e.\ sets.  Suppose that
  there exists $e\in\NN$ such that $\abs{[W_e]_E}>2$ and, for all
  $e'\mathrel{E}^{ce}e$ we have $W_e\subset W_{e'}$.  Then $E$ is not
  induced by any action which is computable in the indices.
\end{thm}

\begin{proof}
  Suppose to the contrary that $\Gamma$ is a computable group acting
  on the c.e.\ sets and that $E$ is the induced orbit equivalence
  relation.  Let $\alpha$ be the computable function which witnesses
  that the action of $\Gamma$ is computable in the indices.  Then for
  each $\gamma$, the map $e\mapsto \alpha(\gamma,e)$ is well-defined
  on c.e.\ sets, and hence monotone.  Now, choose
  $e'\mathrel{E}^{ce}e$ such that $W_e\subsetneq W_{e'}$ and choose
  $\gamma\in\Gamma$ such that $\gamma W_{e'}=W_{e}$.  Now $\gamma
  W_e\subset\gamma W_{e'}=W_{e}$, so by hypothesis, $\gamma W_e=W_e$
  as well.  This contradicts that $\alpha$ really gives rise to an
  action of $\Gamma$, since elements of a group must act by injective
  functions.
\end{proof}

\begin{cor}
  $E_0^{ce}$ is not induced by any action which is computable in the
  indices.
\end{cor}

\begin{proof}
  The empty set $W_e=\emptyset$ is minimal in its $E_0$ equivalence
  class, which consists of all the finite sets.
\end{proof}

This leaves open the following important question.

\begin{question}
  Is $E_0^{ce}$ computably \emph{bireducible} with an orbit relation
  induced by an action which is computable in the indices?
\end{question}

We close this section by showing that like the countable Borel
equivalence relations, the orbit relations induced by actions which
are computable in the indices admit a universal element.  This gives
some hope that the structure of the orbit equivalence relations on
c.e.\ sets will mirror that of the countable Borel equivalence
relations.

\begin{prop}
  \label{prop:universal}
  There exists an equivalence relation $E_\infty^{ce}$ which is
  induced by an action that is computable indices, and satisfies
  $E_\Gamma\leq E_\infty^{ce}$ whenever $E_\Gamma$ arises from an
  action which is computable in the indices.
\end{prop}

\begin{proof}
  We begin by showing that for any computable group $\Gamma$ there
  exists an equivalence relation $U_\Gamma$ which is universal for
  equivalence relations induced by actions of $\Gamma$ which are
  computable in the indices.  For this, we first regard c.e.\ sets as
  codes for functions $\phi\from\Gamma\into CE$ (say by viewing them
  as subsets of $\Gamma\times\NN$ and declaring $\phi(\gamma)=$ the
  $\gamma\th$ column).  We then let $U_\Gamma$ be the equivalence
  relation induced by the action
  $(\gamma\cdot\phi)(g)=\phi(g\gamma^{-1})$.  It is easy to see that
  this action is computable in the indices.  Moreover, if $E_\Gamma$
  is the orbit relation induced by some other action of $\Gamma$ which
  is computable in the indices, then $E_\Gamma$ is reducible to
  $U_\Gamma$ via the function $f(e)=$ a program which enumerates the
  function $\phi_e\from g\mapsto gW_e$.  (Indeed, just check that
  $\gamma W_e=W_{e'}$ if and only if $\gamma^{-1}\cdot\phi_e=\phi_{e'}$.)

  Now, let $F_\omega$ denote the free group on generators
  $x_1,x_2,\ldots$.  We claim that we can take $E_\infty^{ce}$ to be
  $U_{F_\omega}$.  To see this, let $E_\Gamma$ be the orbit
  equivalence relation induced by some action of $\Gamma$ which is
  computable in the indices.  We can regard this as an action of
  $F_\omega$ by simply enumerating $\Gamma=\set{\gamma_i\mid i\in\NN}$
  and letting a word $w(x_1,\ldots x_n)$ act by the composition
  $w\circ(\gamma_1,\ldots,\gamma_n)$.  It is easily seen that this
  action is computable in the indices, and so by the previous
  paragraph $E$ is computably reducible to $U_{F_\omega}$.
\end{proof}

Figure~\ref{fig:enum} shows the two new classes defined in this
section, and the handful of equivalence relations we have considered.

\begin{figure}[ht]
\begin{tikzpicture}
  \draw[dotted] (0,0) circle (3);
  \node[fill=white] at (240:3) {enumerable};
  \draw[dotted] (0,0) circle (1.5);
  \node[fill=white] at (240:1.5) {orbit};
  \node at (-.5,-.5) (eq) {$=^{ce}$};
  \node at (-2,1) (e0) {$E_0^{ce}$} edge (eq);
  \node at (.5,.5) (egam) {$E_\Gamma^{ce}$};
  \node[fill=white] at (90:1.5) {$U_{F_\omega}$} edge (eq) edge (egam);
  \node at (0,4) (eset) {$E_\sset^{ce}$} edge (0,3);
\end{tikzpicture}
\caption{The enumerable relations.  Note that we do not know whether
  $E_0$ is bireducible with an orbit equivalence
  relation.\label{fig:enum}}
\end{figure}
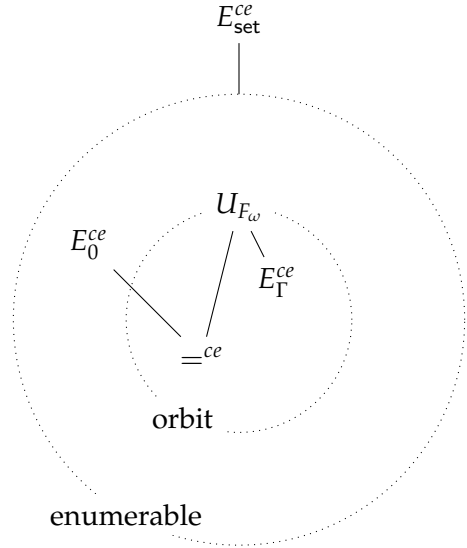

The results of this section have only scratched the tip of the
iceberg.  It would be very interesting to investigate the structure of
the orbit equivalence relations on c.e.\ sets in further detail.  For
instance, we leave the following sample questions.

\begin{question}
  Is $=^{ce}$ minimum among the relations induced by actions which are
  computable in the indices?
\end{question}

\begin{question}
  Does there exist an infinite antichain of relations induced by
  actions which are computable in the indices?
\end{question}

\section{Classification of c.e.\ structures}

We now direct our study towards what we expect will be one of the most
important applications---isomorphism relations on classes of c.e.\
structures.  As mentioned in the introduction, the isomorphism
relations play a prominent role in Borel equivalence relations; in
fact Borel reducibility was initially defined just for isomorphism
relations on classes of countable structures.  This has motivated
several authors to consider various notions of computable reducibility
between classes of countable structures.  Here we begin to use the
machinery built in earlier sections to study classes of c.e.\
structures.

\begin{defn}
  Let $\oiso_\sbin^{ce}$ denote the isomorphism relation on the codes
  for c.e.\ binary relations.  That is, let $e\iso_\sbin^{ce}e'$ if
  and only if $W_e$ and $W_{e'}$, thought of as binary relations on
  $\NN$, are isomorphic.
\end{defn}

We remark that in order to analyze the isomorphism on arbitrary
${\mathcal L}$-structures, it is enough to consider just the binary
relations, since if $\mathcal L$ is a computable language then the
isomorphism relation $\oiso_{\mathcal L}^{ce}$ on the c.e.\ $\mathcal
L$-structures is computably reducible to $\oiso_\sbin^{ce}$.  This
follows from the proof of Proposition~\ref{prop:top}, cited below.

In analogy with the Borel theory, we can study the classification
problem for c.e.\ undirected graphs, trees, linear orders, groups, and
so on by considering the restriction of $\oiso_\sbin^{ce}$ to the
class of indices for such structures.  Unfortunately, we immediately
confront the difficulty that these restrictions are not total, and so
far we have not addressed reducibility for relations which are not
defined on all of $\NN$.

In many practical situations, we can work around this difficulty.  For
instance, rather than identify graphs with binary relations on $\NN$,
we can identify them with subsets of a fixed computable copy $\Gamma$
of the random graph.  Thus, we formally define $\oiso_\sgraph^{ce}$ to
be isomorphism relation on the c.e.\ subsets of $\Gamma$.  These two
coding methods yield equivalent results in following sense: there is a
computable reduction from $\oiso_\sgraph^{ce}$ to $\oiso_\sbin^{ce}$
taking values in the indices for undirected graphs, and there is a
computable function $f$ such that whenever $W_e,W_{e'}$ code
undirected graphs then $W_e\iso_\sbin W_{e'}$ if and only if
$W_{f(e)}\iso_\sgraph W_{f(e')}$ (the last fact follows from
Proposition~\ref{prop:top} below).  Similarly, we can consider the
relation $\oiso_\slo^{ce}$ on the c.e.\ subsets of $\QQ$, and
$\oiso_\stree^{ce}$ on the downward closure of c.e.~subsets of
$\NN^{<\NN}$.

\begin{prop}
  \label{prop:top}
  $\oiso_\sbin^{ce}$ is computably bireducible with each of
  $\oiso_\sgraph^{ce}$, $\oiso_\slo^{ce}$, and $\oiso_\stree^{ce}$.\qed
\end{prop}

The point is that the usual reductions go through in our context as
well (see \cite{gao} for an elegant presentation).  Intuitively, this
is because the reductions only make use of the \emph{positive}
information about the structures; that is, they need to know when two
elements are related, but not when two elements are non-related.

In terms of our hierarchy of equivalence relations, $\oiso_\sbin^{ce}$
is very high.  For one thing, it follows from results of
\cite{knight:iso} that it is $\Sigma^1_1$-complete.  It also lies
above most of the relations considered thus far; as an example we show
the following:

\begin{prop}
  $E_\sset^{ce}$ lies properly below $\oiso_\sbin^{ce}$.
\end{prop}

\begin{proof}
  To reduce $E_\sset^{ce}$ to $\oiso_\sbin^{ce}$, we simply let
  $W_{f(e)}$ be an code for $W_e$, thought of as a hereditarily
  countable set.  In other words, $W_{f(e)}$ is a tree coding the
  transitive closure of $W_e$.  The absence of any reverse reduction
  follows from complexity, since $E_\sset^{ce}$ is $\Pi^0_4$.
\end{proof}

We close our discussion of isomorphism relations by considering the
case of groups.  Classically, the isomorphism relation for
countable groups is Borel bireducible with the other relations
addressed in Proposition~\ref{prop:top}.  However, in our case there
are once again several possible coding methods.  First, we can let
\emph{group} be the set of indices $e$ such that $W_e$, thought of as
a subset of $\NN\times\NN\times\NN$, satisfies the laws for the
multiplication function for a group.  We then define
$\oiso_\sgroup^{ce}$ to be the restriction of the isomorphism relation
$\oiso_\stern^{ce}$ on the c.e.\ ternary relations to
\emph{group}. (Note that all of the elements of \sgroup\ are in fact
computable, because they are c.e.\ functions.)

Alternatively we can code a group by a presentation, that is, a set of
words in $F_\omega$, thinking of the group as the corresponding
quotient. We thus let $e\iso_\spres^{ce}e'$ if and only if $W_e$ and
$W_{e'}$, thought of as sets of relations in $F_\omega$, determine
isomorphic groups. (Note that all groups with c.e.\ presentations
actually have computable presentations by Craig's trick.) This
relation has the advantage of being defined everywhere, but it does
not reflect the same classification problem as $\oiso_\sgroup^{ce}$.
In fact the classification problem for groups splits into two separate
problems: that for computable group multiplication functions, and that
for computably presented groups.

\begin{prop}
  $\oiso_\sgroup^{ce}\leq\oiso_\sbin^{ce}\leq\oiso_\spres^{ce}$.
\end{prop}

We suspect that neither reduction is reversible.

\begin{proof}
  Of course $\oiso_\sgroup^{ce}$ is computably reducible to
  $\oiso_\stern^{ce}$ in the sense that there is a computable function
  (the identity) which, when restricted to \emph{group}, satisfies
  $e\iso_\sgroup^{ce}e'$ if and only if $f(e)\iso_\stern^{ce}f(e')$.
  Hence, $\oiso_\sgroup^{ce}$ is reducible to $\oiso_\sbin^{ce}$ in
  the same sense.  To see that $\oiso_\sbin^{ce}\leq\oiso_\spres^{ce}$
  one need only inspect the classical argument of \cite{mekler}, which
  yields group presentations, and check that it can be done in our
  context as well.
\end{proof}

We next consider the \emph{computable isomorphism} equivalence
relations on these classes of structures.

\begin{defn}
  Let $\ociso_\sbin^{ce}$ denote the computable isomorphism relation
  on the space of c.e.\ binary relations.  That is
  $e\ciso_\sbin^{ce}e'$ if and only if $W_e$ and $W_{e'}$, thought of
  as codes for binary relations on $\NN$, are isomorphic via a
  computable bijection.
\end{defn}

The results of Proposition~\ref{prop:top} apply also to the case of
computable isomorphism.  For instance if $\phi$ is a computable
isomorphism between $W_e$ and $W_{e'}$, and $f$ is the reduction given
in \cite{gao} from binary relations to to graphs, then it is easy to
use $\phi$ to find a computable isomorphism between $W_{f(e)}$ and
$W_{f(e')}$.  This need not always hold, since sometimes one is able
to show that $W_e\iso W_{e'}$ if and only if $W_{f(e)}\iso W_{f(e')}$
without necessarily constructing the isomorphisms explicitly.  But it
is not difficult to check that it does hold for the examples in this
section.

\begin{prop}
  The computable isomorphism relation on the class of c.e.\ binary
  relations, graphs, linear orders and trees are all computably
  bireducible.\qed
\end{prop}

On the other hand, the computable isomorphism relation lies much lower
in the hierarchy than the isomorphism relation.

\begin{prop}
  $\ociso_\sbin^{ce}$ lies properly below $E_\sset^{ce}$.
\end{prop}

\begin{proof}
  Let $W_e\subset\NN$ be a c.e.\ set, which we shall think of as
  coding a binary relation.  We will create a c.e.\ subset $W_{f(e)}$
  of $\NN\cup{-1}\times\NN$ whose columns consist of all computably
  isomorphic copies of $W_e$ whose domain is a subset of $\NN$,
  together with all finite sets of the form $F\cup{-1}$.  To do this,
  $f(e)$ first arranges to write all finite sets of the form
  $F\cup{-1}$ onto the odd-numbered columns.  Then, it considers all
  pairs of programs $p,p'$, hoping in each case that $W_p$ codes a
  bijection $\phi_p$ of $\NN$ with itself and $W_{p'}$ codes its
  inverse.  As $p$ is simulated, we write $\phi_p$ applied to the
  graph $W_e$ onto the column $2n_{p,p'}$ (where $n_{p,p'}$ is a code
  for the pair $(p,p')$).  If $W_p$ and $W_{p'}$ do not turn out to
  code a bijection and its inverse then this will be apparent at some
  stage of simulation, and at that point we enumerate $-1$ into column
  $2n_{p,p'}$ and then stop writing to that column.  It is not
  difficult to check that this reduction is as desired.

  Finally, to see that $E_\sset^{ce}$ is not reducible to
  $\ociso_\sbin^{ce}$ simply compute that $\ociso_\sbin^{ce}$ is
  $\Sigma^0_3$, but it follows from Theorem~\ref{thm:e3} that in
  $E_\sset^{ce}$ there is a $\Pi^0_3$ complete equivalence class.
\end{proof}

The hierarchy of isomorphism relations considered in this section is
shown in Figure~\ref{fig:iso}.

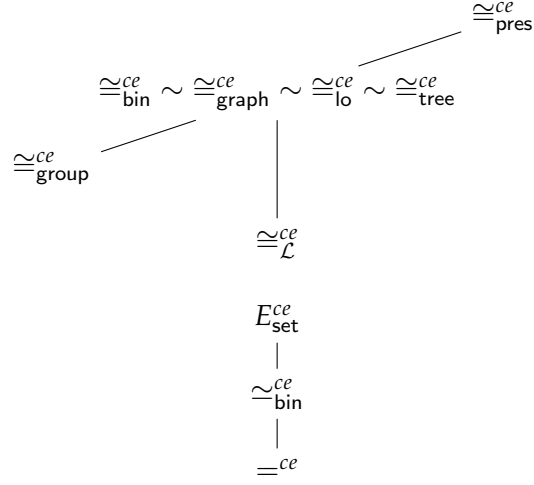
\begin{figure}[ht]
\begin{tikzpicture}
  \node at (0,0) (eq) {$=^{ce}$};
  \node at (0,1) (comp) {$\ociso_\sbin^{ce}$} edge (eq);
  \node at (0,2) (set) {$E_\sset^{ce}$} edge (comp);
  \node at (0,3) (isol) {$\oiso_{\mathcal L}^{ce}$};
  \node at (0,5) (iso) {$\oiso_\sbin^{ce}\sim\oiso_\sgraph^{ce}\sim\oiso_\slo^{ce}\sim\oiso_\stree^{ce}$} edge (isol);
  \node at (-3,4) {$\oiso_\sgroup^{ce}$} edge (iso);
  \node at (3,6) {$\oiso_\spres^{ce}$} edge (iso);
\end{tikzpicture}
\caption{Relationships between the isomorphism equivalence relations
  considered in this paper.\label{fig:iso}}
\end{figure}

\section{Relations from computability theory}

Some of the most important examples of relations on c.e.\ sets are
those arising from computability theory itself.  In this section, we
consider the degree-theoretic equivalence relations, fitting them into
the computable reducibility hierarchy.  We then briefly generalize our
method of dealing with the c.e.\ sets to handle the larger class of
$n$-c.e.\ sets.

We begin with the degree-theoretic equivalence relations.

\begin{thm}
  $=^{ce}$ lies properly below each of $\oequiv_T^{ce}$,
  $\oequiv_1^{ce}$, and $\oequiv_m^{ce}$ in the computable
  reducibility hierarchy.
\end{thm}

\begin{proof}
  We will define a function which reduces $=^{ce}$ to all three
  relations at once.  To begin, we use a strong form of the
  Friedberg-Muchnik theorem to obtain a c.e.\ sequence of sets $A_i$
  with the property that for all $i$, we have
  $A_i\not\leq_T\bigoplus_{j\neq i}A_j$ (and also $\not\leq_1$ and
  $\not\leq_m$).  Then, we let $W_{f(e)}$ be the subset of
  $\NN\times\NN$ whose $k\th$ column is $A_{n_k}$, where $n_k$ is the
  $k\th$ element enumerated into $W_e$ by $e$.

  We first show that if $W_e=W_{e'}$ then $W_{f(e)}\equiv_1 W_{f(e')}$
  (and hence $\oequiv_m$ and $\oequiv_T$).  Assuming first that $W_e$
  is infinite, $W_{f(e)}$ can be obtained from $W_{f(e')}$ by the
  following permutation of $\NN\times\NN$: If the $k\th$ element to
  appear in $W_{e'}$ is the $r\th$ element to appear in $W_e$, then
  send the $k\th$ column to the $r\th$ column.  In the case that $W_e$
  is finite, $W_{f(e)}$ can be obtained by $W_{f(e')}$ by a finite
  permutation of the columns.

  We now show the converse: that if $W_e\neq W_{e'}$ then
  $W_{f(e)}\not\equiv_TW_{f(e')}$ (and $\not\equiv_1$ and
  $\not\equiv_m$).  For this, we can suppose that $i\in
  W_e\smallsetminus W_{e'}$.  Now, suppose towards a contradiction
  that $W_{f(e)}\leq_TW_{f(e')}$.  Then we have
  \[A_i\;\leq_T\;W_{f(e)}\;\leq_T\;W_{f(e')}\;\leq_T\;\bigoplus_{j\neq i}A_i\;,
  \]
  the last reduction holding using arguments similar to the previous
  paragraph.  But this contradicts our choice of the $A_i$.  As noted,
  this argument also works for $\leq_1$ and $\leq_m$.

  Finally, none of the three degree relations are reducible to
  $=^{ce}$ because they are each $\Sigma^0_3$ complete (see for
  instance \cite[Corollary~IV.3.6]{soare}) while $=^{ce}$ is just
  $\Pi^0_2$.
\end{proof}

Note that in defining the set $W_{f(e)}$ we cannot simply use the sum
$\bigoplus_{j\in W_e}A_j$, since here a complicated set is coded into
the indices of summation.

\begin{thm}
  $\oequiv_m^{ce}$ is computably reducible to $\oequiv_1^{ce}$.
\end{thm}

\begin{proof}
  For this we simply let $W_{f(e)}=W_e\times\NN$.  If
  $\phi\from\NN\into\NN$ is a many-one reduction from $W_e$ to
  $W_{e'}$, then the map $(m,n)\mapsto(\phi(m),\langle m,n\rangle)$ is
  a one-one reduction from $W_{f(e)}$ to $W_{f(e')}$.  Conversely, if
  $\psi\from\NN\times\NN\into\NN\times\NN$ is a one-one reduction from
  $W_{f(e)}$ to $W_{f(e')}$, then the map $m\mapsto$ the first
  coordinate of $\psi(m,0)$ is a many-one reduction from $W_e$ to
  $W_{e'}$.
\end{proof}

The next result clarifies how the degree-theoretic relations fit in
with other relations previously considered.  This and earlier results
are depicted in Figure~\ref{fig:turing}.

\begin{prop}
  $\oequiv_1^{ce}$ is computably reducible to the computable
  isomorphism relation $\ociso_\sbin^{ce}$ on c.e.\ binary relations.
\end{prop}

\begin{proof}
  The point is that $\oequiv_1^{ce}$ is precisely the computable
  isomorphism relation on the set of c.e.\ \emph{unary} relations.  So
  for instance $\oequiv_1^{ce}\leq\ociso_\sbin^{ce}$ via the
  map such that $W_{f(e)}$ codes the graph on $\NN\cup\{\star\}$ where
  $\star\rightarrow n$ if and only if $n\in W_e$.
\end{proof}

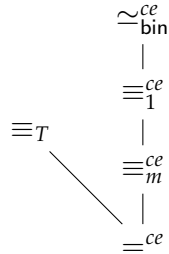
\begin{figure}[ht]
\begin{tikzpicture}
  \node at (0,0) (eq) {$=^{ce}$};
  \node at (0,1) (m) {$\oequiv_m^{ce}$} edge (eq);
  \node at (0,2) (1) {$\oequiv_1^{ce}$} edge (m);
  \node at (0,3) (cbin) {$\ociso_\sbin^{ce}$} edge (1);
  \node at (-1.5,1.5) (T) {$\oequiv_T$} edge (eq);
\end{tikzpicture}
\caption{Relationships between the degree-theoretic equivalence
  relations and some others considered in earlier
  sections.\label{fig:turing}}
\end{figure}

We close this section with a generalization of the reducibility
hierarchy on c.e.\ sets.  Here, we shall consider equivalence
relations on the \emph{d.c.e.}\ and even $n$-c.e.\ sets.  This is
natural given what we have done, because like the c.e.\ sets, the
$n$-c.e.\ sets are also characterized by natural number indices.

\begin{defn}\
  \begin{itemize}
  \item If $\bm{e}=\seq{e_1,\ldots,e_n}$ is a sequence of indices then
    the corresponding $n$-c.e.\ subset of $\NN$ is the set
    $W_{\bm{e}}=(W_{e_1}\smallsetminus W_{e_2})\cup
    (W_{e_3}\smallsetminus W_{e_4})\cup\cdots\smallsetminus\cup
    W_{e_n}$.  Here, the last operation is either $\smallsetminus$ or
    $\cup$ depending on whether $n$ is even or odd.
  \item If $E$ is an equivalence relation on $n$-c.e.\ sets, then
    $E^{n\text{-}ce}$ is the relation on $\NN^n$ defined by
    $\bm{e}\mathrel{E}^{n\text{-}ce}\bm{f}$ if and only if
    $W_{\bm{e}}\mathrel{E}W_{\bm{f}}$.
  \end{itemize}
\end{defn}

The $1$-c.e.\ sets are of course the c.e.\ sets, and the $2$-c.e.\
sets are sometimes called the d.c.e.\ sets (difference of c.e.\
sets). Thus we shall write $E^{dce}$ for $E^{2\text{-}ce}$.  It is
trivial to check that for all $n$ we always have $E^{n\text{-}ce}\leq
E^{n+1\text{-}ce}$: fix an $e$ with $W_e=\emptyset$, and let
$(e_1,\ldots,e_n)\mapsto (e_1,\ldots,e_n,e)$.

%

\begin{thm}
  \label{thm:nce}
  For every $n>0$, $=^{n\text{-}ce}$ lies properly below
  $=^{(n+1)\text{-}ce}$.
\end{thm}

\begin{proof}
  We claim that there is no computable reduction from
  $=^{(n+1)\text{-}ce}$ to $=^{n\text{-}ce}$.  Suppose that $g$ were
  such a reduction.  Fix some $\bm{e}$ and $\bm{f}$ of length $(n+1)$
  with $W_{\bm{e}}=\{0\}$ and $W_{\bm{f}}=\emptyset$.  By assumption,
  $W_{g(\bm{e})}\neq W_{g(\bm{f})}$, and we suppose first that there
  exists some number $m\in W_{g(\bm{e})}\smallsetminus W_{g(\bm{f})}$.
  Of course, the $n$-c.e.\ sets $W_{g(\bm{e})}$ and $W_{g(\bm{f})}$
  can be approximated using the indices $g(\bm{e})$ and $g(\bm{f})$.
  We will use the Recursion Theorem to define an $(n+1)$-c.e.\ set
  $W_{\bm{i}}$ which ``knows its own indices $\bm{i}$'' and is
  approximated as follows.  We will use the notation $W_{\bm{i},s}$
  denote the approximation to $W_{\bm{i}}$ at stage $s$.

  We start with $W_{\bm{i},0}=\emptyset$ and $W_{\bm{i},1}=\{0\}$.  At
  the first stage $s_0$ (if any) with $m\in W_{g(\bm{i}),s_0}$, we
  take $0$ out, leaving $W_{\bm{i},s_0+1}=\emptyset$.  Then we do
  nothing further until the next stage $s_1>s_0$ at which $m\notin
  W_{g(\bm{i}),s_1}$.  At stage $s_1+1$, we enumerate $0$ back into
  $W_{\bm{i}}$, leaving $W_{\bm{i},t+1}=\{0\}$.  Then, if we encounter
  another stage $s_2>s_1$ at which $m$ enters $W_{g(\bm{i})}$ again,
  we take $0$ back out of $W_{\bm{i}}$, and so on, back and forth, as
  many times as $m$ enters or leaves the set $W_{g(\bm{i})}$.  This
  happens at most $n$ times, so $W_{\bm{i}}$ is $(n+1)$-c.e.\
  (including the initial enumeration of $0$ into $W_{\bm{i},1}$).
  However, our construction leaves $W_{\bm{i}}=\emptyset=W_{\bm{f}}$
  if and only if $m\in W_{g(\bm{i})}$, in which case
  $W_{g(\bm{i})}\neq W_{g(\bm{f})}$; whereas, if $m\notin
  W_{g(\bm{i})}$, then $W_{\bm{i}}=\{0\}=W_{\bm{e}}$, yet
  $W_{g(\bm{i})}\neq W_{g(\bm{e})}$, since $m\in W_{g(\bm{e})}$.

  If there is no $m$ in $W_{g(\bm{e})}\smallsetminus W_{g(\bm{f})}$,
  then there must be some number $m'\in W_{g(\bm{f})}\smallsetminus
  W_{g(\bm{e})}$, since $W_{\bm{e}}\neq W_{\bm{f}}$ and $g$ is assumed
  to be a reduction.  In this case, we execute the same construction,
  except that we leave $W_{\bm{i},s}=\emptyset$ until reaching a stage
  $s_0$ with $m'\in W_{g(\bm{i}),s_0}$, then enumerate $0$ into
  $W_{\bm{i}}$, then wait for $m'$ to leave $W_{g(\bm{i})}$, then
  remove $0$ from $W_{\bm{i}}$, and so on.  In this case, $W_{\bm{i}}$
  is actually just $n$-c.e., not $(n+1)$-c.e., and again the strategy
  works: if $m'\in W_{g(\bm{i})}$, then $W_{\bm{i}}=\{0\}=W_{\bm{e}}$,
  yet $m'\notin W_{g(\bm{e})}$; whereas, if $m'\notin W_{g(\bm{i})}$,
  then $W_{\bm{i}}=\emptyset=W_{\bm{f}}$, yet $m'\in W_{g(\bm{f})}$.
  So $g$ cannot have been a computable reduction.
\end{proof}

We can define an equivalence relation above all of these.  Let
$\bm{h}:\omega\to\omega^{<\omega}$ be a computable bijection, and define
\[i \equiv^{<\omega\text{-}ce} j \iff W_{\bm{h}(i)}=W_{\bm{h}(j)}\;.
\]
Here, if $\bm{h}(i)\in\omega^{<\omega}$ has length $n$, then
$W_{\bm{h}(i)}$ is exactly the $n$-c.e.\ set defined above.  So this
relation $=^{<\omega\text{-}ce}$ is really just an amalgam of all the
relations $=^{n\text{-}ce}$.  It is clear that
$=^{n\text{-}ce}~\leq~=^{<\omega\text{-}ce}$ for every $n$, and
Theorem \ref{thm:nce} then shows that
$=^{<\omega\text{-}ce}~\not\leq~=^{n\text{-}ce}$.

Many of the results in this paper concerning equivalence relations on
c.e.\ sets have analogues for the $n$-c.e.\ sets.  However, it is also
interesting to consider how relations on $n$-c.e.\ sets fit in with
the relations on the c.e.\ sets.

\begin{thm}
  The relation $=^{<\omega\text{-}ce}$ is computably reducible to
  $E_3^{ce}$, but no reduction exists in the opposite direction.
\end{thm}

\begin{proof}
  For the reduction, let $\bm{e}=\seq{e_1,\ldots,e_n}$ be an index for
  an $n$-c.e.\ set.  We define a program $f((e_i))$ enumerating a
  subset of $\NN\times\NN$ as follows.  Begin by simulating the
  programs $e_1,\ldots,e_n$.  If $k$ appears in $W_{\bm{e},s}$, we
  write the first $s$ elements of the $k\th$ column into
  $W_{f(\bm{e})}$.  Note that as $s$ increases, the status of $k\in
  W_{\bm{e},s}$ can only change $n$ times, and therefore $k\in
  W_{\bm{e}}$ if and only if the $k\th$ column of $W_{f(\bm{e})}$ is
  infinite.  It follows easily that $f$ is a reduction to $E_3^{ce}$.

  In the opposite direction, we note that the relation $E_3^{ce}$ is
  $\Pi^0_3$-complete.  However, $W_{\bm{e}}=W_{\bm{f}}$ iff, for every
  $x$, we have $x\in W_{\bm{e}}$ if and only if $x\in W_{\bm{f}}$.
  Since membership in each of $W_{\bm{e}}$ and $W_{\bm{f}}$ is
  $\Delta^0_2$, the relation $=^{<\omega-ce}$ is $\Pi^0_2$, precluding
  any computable reduction.
\end{proof}

\begin{figure}[ht]
\begin{tikzpicture}
  \node at (0,0) (eq) {$=^{ce}$};
  \node at (0,1) (eqd) {$=^{dce}$} edge (eq);
  \node at (0,2) (eqn) {$=^{n\text{-}ce}$};
  \node at (0,3) (eqo) {$=^{<\omega\text{-}ce}$};
  \node at (0,4) (e3) {$E_3^{ce}$};
  \draw[-] (e3) -- (eqo); 
  \draw[-] (eqo) -- (eqn); 
  \draw[-] (eqn) -- (eqd); 
  \node at (-1.5,2) (e0) {$E_0^{ce}$} edge (eq);
  \draw[-] (e3) -- (e0);
\end{tikzpicture}
\caption{Known relationships between the relations on $n$-c.e.\
  sets.\label{fig:nce}}
\end{figure}
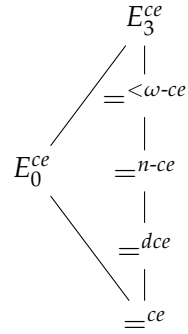

These relationships are shown in Figure~\ref{fig:nce}.  It would also
be interesting to decide the relationship between $=^{n\text{-}ce}$
and $E_0^{ce}$.  Since $E_0^{ce}$ is is $\Sigma^0_3$ complete while
$=^{n\text{-}ce}$ is just $\Pi^0_2$, we can conclude that $E_0^{ce}$
is not computably reducible to $=^{n\text{-}ce}$.  This leaves open
the following question.

\begin{question}
  Are any of the $=^{n\text{-}ce}$ computably reducible to $E_0^{ce}$?
\end{question}

\bibliographystyle{alpha}
\begin{singlespace}
  \bibliography{compequiv}
\end{singlespace}

\end{document}